\pdfoutput=1
\documentclass[10pt,a4paper,oneside]{amsart}
\usepackage{amsmath, amssymb, amsfonts}
\usepackage[a4paper]{geometry}

\numberwithin{equation}{section}

\usepackage{stmaryrd}
\newcommand{\llb}{\llbracket}
\newcommand{\rrb}{\rrbracket}

\DeclareMathOperator{\diag}{diag}
\DeclareMathOperator{\End}{End}
\DeclareMathOperator{\Frac}{Frac}
\DeclareMathOperator{\Gal}{Gal}
\DeclareMathOperator{\GL}{GL}

\DeclareMathOperator{\Hom}{Hom}

\DeclareMathOperator{\Irr}{Irr}
\DeclareMathOperator{\nr}{nr}
\DeclareMathOperator{\ord}{ord}

\DeclareMathOperator{\res}{res}
\DeclareMathOperator{\Quot}{Quot}

\DeclareMathOperator{\Tr}{Tr}
\DeclareMathOperator{\tr}{tr}
\DeclareMathOperator{\Jac}{Jac}
\DeclareMathOperator{\Ext}{Ext}

\newcommand{\mm}{\mathfrak m}
\newcommand{\pp}{\mathfrak p}

\newcommand{\G}{\mathcal G}
\newcommand{\OO}{\mathcal O}
\newcommand{\Q}{\mathcal Q}
\usepackage[scr=dutchcal]{mathalfa}

\newcommand{\QQ}{\mathbb Q}

\newcommand{\ZZ}{\mathbb Z}
\DeclareSymbolFont{bbold}{U}{bbold}{m}{n}
\DeclareSymbolFontAlphabet{\mathbbold}{bbold}
\newcommand{\LLambda}{\mathbbold \Lambda}
\newcommand{\GGamma}{\mathbbold \Gamma}
\usepackage{bbm}

\renewcommand{\epsilon}{\varepsilon}
\renewcommand{\phi}{\varphi}
\renewcommand{\theta}{\vartheta}
\newcommand{\e}{\mathsf e}
\newcommand{\al}{\mathrm{c}}
\newcommand{\cent}{\mathfrak z}

\newcommand{\OF}{{\OO_F}}

\newcommand{\D}{\mathsf D}
\newcommand{\DD}{\mathfrak D}
\newcommand{\M}{\mathfrak M}
\newcommand{\F}{\mathcal F}
\usepackage{tikz-cd}
\usepackage{colonequals}

\newenvironment{psmallmatrix}
{\left(\begin{smallmatrix}}
	{\end{smallmatrix}\right)}

\usepackage{tikz}
\usetikzlibrary{cd}
\usetikzlibrary{arrows}
\usetikzlibrary{positioning}
\usetikzlibrary{calc}

\usepackage[style=alphabetic,backend=biber,isbn=false,url=false]{biblatex}
\usepackage{enumitem}
\newlist{theoremlist}{enumerate}{1}
\setlist[theoremlist]{label=(\roman{theoremlisti}), ref=\thetheorem.\roman{theoremlisti},noitemsep, topsep=.2ex}
\newlist{propositionlist}{enumerate}{1}
\setlist[propositionlist]{label={(\roman{propositionlisti})}, ref=\theproposition.\roman{propositionlisti},noitemsep, topsep=.2ex}
\newlist{conjecturelist}{enumerate}{1}
\setlist[conjecturelist]{label={(\roman{conjecturelisti})}, ref=\theconjecture.\roman{conjecturelisti},noitemsep, topsep=.2ex}
\newlist{lemmalist}{enumerate}{1}
\setlist[lemmalist]{label=(\roman{lemmalisti}), ref=\thelemma.\roman{lemmalisti},noitemsep, topsep=.2ex}
\newlist{definitionlist}{enumerate}{1}
\setlist[definitionlist]{label=(\roman{definitionlisti}), ref=\thelemma.\roman{definitionlisti},noitemsep, topsep=.2ex}

\usepackage{etoolbox} % needed below

\usepackage{mathtools}

\usepackage[hidelinks,plainpages=false,pdfpagelabels]{hyperref}
\usepackage[capitalise, noabbrev]{cleveref}

\theoremstyle{plain}
\newtheorem{theorem}{Theorem}[section]
\newtheorem{lemma}[theorem]{Lemma}

\newtheorem{proposition}[theorem]{Proposition}
\newtheorem{corollary}[theorem]{Corollary}
\newtheorem{conjecture}[theorem]{Conjecture}

\newtheorem*{theorem*}{Theorem}
\newtheorem*{conjecture*}{Conjecture}
\theoremstyle{definition}
\newtheorem{definition}[theorem]{Definition}
\newtheorem{notation}[theorem]{Notation}
\theoremstyle{remark}
\newtheorem{remark}[theorem]{Remark}
\newtheorem{example}[theorem]{Example}

\Crefname{lemma}{Lemma}{Lemmata}
\Crefname{claim}{Claim}{Claims}
\Crefname{proposition}{Proposition}{Propositions}
\Crefname{conjecture}{Conjecture}{Conjectures}
\Crefname{example}{Example}{Examples}
\crefname{page}{page}{pages}
\Crefname{condition}{Condition}{Conditions}
\Crefname{question}{Question}{Questions}
\Crefname{theoremlisti}{Theorem}{Theorems}
\Crefname{propositionlisti}{Proposition}{Propositions}
\Crefname{conjecturelisti}{Conjecture}{Conjectures}
\Crefname{lemmalisti}{Lemma}{Lemmata}
\Crefname{definitionlisti}{Definition}{Definitions}

\NewDocumentEnvironment{noproof}{m}{% #1 is the inner environment
  \par\pushQED{\qed}\UseName{#1}%
}{\popQED\UseName{end#1}}

\title{Graduated orders over completed group rings and conductor formul\ae}
\author{Ben Forrás}
\address{Universität der Bundeswehr\\
	INF 1 Institut für Theoretische Informatik, Mathematik und Operations Research \\
	Werner-Heisenberg-Weg 39\\
	85579 Neubiberg\\
	Germany}
\email{ben.forras@unibw.de}
\urladdr{https://bforras.eu}

\subjclass[2020]{16H10, 16H20, 11R23}
\keywords{graduated orders, central conductor, completed group algebras}
\date{Version of 2025-10-06}
\begin{document}

\begin{abstract}
We study graduated orders over completed group rings of {one}-dimensional admissible $p$-adic Lie groups, and verify the equivariant $p$-adic Artin conjecture for such orders. Following Jacobinski and Plesken, we obtain a formula for the conductor of a graduated order into a self-dual order. We also refine Nickel's central conductor formula by determining a hitherto implicit exponent $r_\chi$.
\end{abstract}

\maketitle

\section{Introduction}
\subsection{Overview of conductor formul\ae}

For a ring extension $R\subseteq R'$, the \emph{left resp. right conductors} are defined as
    \[(R':R)_\ell \colonequals\{x\in R' : xR'\subseteq R\} \quad\text{resp.}\quad (R':R)_r\colonequals\{x\in R': R'x\subseteq R\}.\]
The initial interest in conductors stemmed from the fact that they can be used to bound certain $\Ext^1$-groups, see \cite{Jacobinski} and \cite[\S29]{CR}. Another, more recent application is in the study of Fitting invariants, see \cite[\S6]{ncfitt2}.

In the case of a maximal order {containing} a group ring, the left/right conductor {is independent of the choice of this maximal order, and it} has been determined by Jacobinski \cite{Jacobinski}:
\begin{theorem}[Jacobinski, {\cite[Theorem~27.8]{CR}}] \label{thm:Jacobinski-conductor}
    Let $H$ be a finite group, and let $R$ be a Dedekind domain with field of fractions $L=\Frac(R)$ {such that the characteristic of $L$ is coprime to $\#H$}. Let $\GGamma$ be a maximal $R$-order in the group algebra $L[H]$ containing $R[H]$. Then
    \[(\GGamma:R[H])_\ell = (\GGamma:R[H])_r = \bigoplus_{\eta\in\Irr(H)/\sim_L} \left(\frac{\#H}{ \eta(1) }\right) \D(\GGamma \epsilon(\eta) / R).\]
\end{theorem}
Here $\eta$ runs through irreducible characters of $H$ up to Galois equivalence: $\eta\sim_L\eta'$ if there exists $\sigma\in\Gal(L(\eta)/L)$ such that $\sigma\circ \eta=\eta'$, where $L(\eta)=L(\eta(h):h\in H)$. The idempotent $\epsilon(\eta)$ is defined in \eqref{eq:epsilon-idempotents}, and $\D$ denotes the inverse different (\cref{sec:inverse-different}).

While this result is usually cited only for maximal orders, it also holds for hereditary orders \cite[Satz~10.7]{JacobinskiEssen}.
A further generalisation is due to Plesken \cite[Theorem~III.8]{Plesken}, who replaced $\GGamma$ by a graduated order and $R[H]$ by a self-dual order; see \cref{def:graduated,def:self-dual} for these notions.
Jacobinski's result has been adapted to completed group rings by Nickel {as follows.}

{From now on, let $p$ always denote an odd rational prime.} {By a one-dimensional admissible $p$-adic Lie group we mean a one-dimensional $p$-adic Lie group $\G$ containing a finite normal subgroup $H$ such that $\G/H\simeq \ZZ_p$; then $\G=H\rtimes \Gamma$ where $\Gamma\simeq\ZZ_p$, see \cite[551]{TEIT-II}.}
\begin{theorem}[{\cite[Theorem~2.8]{NickelConductor}}] \label{thm:Nickel-conductor}
    Let $\G=H\rtimes\Gamma$ be a $p$-adic Lie group of dimension $1$, where $H$ is a finite group and $\Gamma\simeq\ZZ_p$. Let $\Gamma_0$ be a central subgroup of $\G$ isomorphic to $\ZZ_p$. Let $F/\QQ_p$ be a finite extension, $\Lambda^{\OO_F}(\G)\colonequals \OO_F\llb \G\rrb$ a completed group ring, and $\Q^F(\G)\colonequals\Quot(\Lambda^{\OO_F}(\G))$ its total ring of quotients.
    Let $\M$ be a maximal $\Lambda^{\OO_F}(\Gamma_0)$-order in $\Q^F(\G)$ containing $\Lambda^{\OO_F}(\G)$. Then 
    \[\left(\M:\Lambda^{\OO_F}(\G)\right)_\ell = \left(\M:\Lambda^{\OO_F}(\G)\right)_r = \bigoplus_{\chi\in\Irr(\G)/\sim_F} \left(\frac{\#H}{\chi(1)}\right) \D_\chi(\M).\]
\end{theorem}
The relevant notation shall be recalled in \cref{sec:maximal}. 
We remark that it doesn't make sense to consider this question for hereditary orders: as we will show in \cref{no-hereditary-orders}, there \emph{are} no hereditary orders in this setting (see also \cref{sec:hereditary}).

One also considers the central conductor, which is defined as follows.
Let $R$ be a noetherian integrally closed integral domain with field of fractions $L=\Frac(R)$, and let $\mathscr A$ be a separable $L$-algebra with centre $\cent(\mathscr A)\supseteq L$. Let $R^{\mathrm{int}}_{\cent(\mathscr A)}$ denote the integral closure of $R$ in $\cent(\mathscr A)$. Let $\LLambda\subseteq \mathscr \GGamma\subset \mathscr A$ be two $R$-orders. The \emph{central conductor} of $\GGamma$ into $\LLambda$ is defined as
\begin{equation} \label{eq:def-central-conductor}
    \F(\GGamma/\LLambda) \colonequals R^{\mathrm{int}}_{\cent(\mathscr A)} \cap (\GGamma:\LLambda)_\ell = R^{\mathrm{int}}_{\cent(\mathscr A)} \cap (\GGamma:\LLambda)_r = \left\{x\in R^{\mathrm{int}}_{\cent(\mathscr A)} : x \GGamma\subseteq \LLambda \right\}.
\end{equation}
\begin{theorem}[Jacobinski, {\cite[Theorem~27.13(a)]{CR}}] \label{thm:Jacobinski-central-conductor}
    In the setup of \cref{thm:Jacobinski-conductor}, one has
    \[\F(\GGamma / R[H]) = \bigoplus_{\eta\in\Irr(H)/\sim_K} \left(\frac{\#H}{\eta(1)}\right) \D\left(R^{\mathrm{int}}_{\cent(\mathscr A)}\epsilon(\eta) \big/ R\right);\]
    in particular, the central conductor does not depend on the choice of the maximal order $\GGamma$.
\end{theorem}
For completed group rings, in the setup of \cref{thm:Nickel-conductor}, Nickel showed \cite[Theorem~3.5]{NickelConductor} that the central conductor $\F\left(\M / \Lambda^{\OO_F}(\G)\right)$ doesn't depend on the choice of the maximal order $\mathfrak M$, and there is an inclusion
\begin{equation} \label{eq:Nickel-central-conductor}
    \F\left(\M / \Lambda^{\OO_F}(\G)\right) \supseteq  \bigoplus_{\chi\in\Irr(\G)/\sim_F} \left(\frac{\#H w_\chi}{\chi(1)}\right) \D(\OO_{F_\chi}/\OO_F) \Lambda^{\OO_{F_\chi}}(\Gamma'_\chi).
\end{equation}
Here $w_\chi$ is the index of the stabliliser of an irreducible constituent of $\res^\G_H\chi$, 
$\OO_{F_\chi}$ is the ring of integers in $F_\chi\colonequals F(\chi(h):h\in H)$,
and $\Lambda^{\OO_{F_\chi}}(\Gamma'_\chi)=\OO_{F_\chi}\llb\Gamma'_\chi\rrb$ is the Iwasawa algebra over a certain group $\Gamma'_\chi\simeq \ZZ_p$ defined by Nickel.

The containment \eqref{eq:Nickel-central-conductor} is an equality whenever (i) $\G\simeq H\times \Gamma$ is a direct product, or (ii) the Schur index $s_\chi$ is not divisible by $p$. 
For pro-$p$ groups $\G$ (that is, in the case of $p$-groups $H$) and $F=\QQ_p$, Nickel used Lau's description \cite{Lau} of the Wedderburn decomposition of $\Q(\G)$ to compute the factor missing from the right hand side, making the containment an equality. 
We will determine the missing factor for general $\G$ in \cref{sec:central-conductor}; see \cref{rem:comparison} for a comparison with Nickel's results.

\subsection{New results}
The goal of this article is to extend some of the results above. 
In \cref{sec:graduated}, we consider graduated orders beyond the well-studied case of Dedekind domains, which we shall refer to as `the {one}-dimensional case'. Modifying Plesken's arguments appropriately, we obtain the following conductor formula (the setup and notation shall be explained in \cref{sec:graduated-higher-dim}):
\begin{theorem}[{\cref{general-conductor-formula}}] \label{thm:cf-graduated}
    Let {$R$ be a regular local ring of dimension at most two with field of fractions $L$, and let $\mathscr A$ be a separable $L$-algebra.}
    Let $\LLambda\subset \mathscr A$ be a self-dual $R$-order, and let $\LLambda\subseteq \GGamma\subset \mathscr A$ be a graduated $R$-overorder of $\LLambda$. Then
    \[(\GGamma:\LLambda)_r = (\GGamma:\LLambda)_\ell = \D(\GGamma/R) \simeq \bigoplus_{\chi \in \mathscr X} \LLambda\left(\mathbf n_\chi, \D(\Omega_\chi/R)\mathbf E_\chi * \mathbf I_\chi^{-,\top}\right).\]
    In particular, the conductor is independent of $\LLambda$.
\end{theorem}
{In a recent work of the author \cite{W}, the Wedderburn decomposition of the total ring of quotients of the completed group ring over a {one}-dimensional admissible $p$-adic Lie group is described. In particular, this shows that the algebras $\Q^F(\G)$ {satisfy the hypotheses} of \cref{thm:cf-graduated}.}
In \cref{sec:maximal}, we reformulate Nickel's work in the language of \cite{W}, which essentially means replacing the profinite group $\Gamma'_\chi$ by another profinite group $\Gamma''_\chi$ (both groups are abstractly isomorphic to $\ZZ_p$). The description of the Wedderburn decomposition of $\Q^F(\G)$ then leads to a more explicit formula for the central conductor:
\begin{theorem}[{\S\S\ref{sec:conductor-formula}--\ref{sec:central-conductor}}] \label{thm:cf-max}
    Nickel's conductor formula {(\cref{thm:Nickel-conductor})} and central conductor formula {\eqref{eq:Nickel-central-conductor}} remain valid under replacing $\Gamma'_\chi$ with $\Gamma''_\chi$. Moreover, the central conductor is
    \[\F\left(\M / \Lambda^{\OO_F}(\G)\right) = \bigoplus_{\chi\in\Irr(\G)/\sim_F} \left(\frac{\#H w_\chi}{\chi(1)}\right) \D(\OO_{F_\chi}/\OO_F) \left(\pp_\chi'\right)^{r_\chi},\]
    where $\pp'_\chi$ is the height $1$ prime ideal of the local ring $\Lambda^{\OO_{F_\chi}}\llb\Gamma'_\chi\rrb$ above $p$, and the exponent $r_\chi$ is determined explicitly.
\end{theorem}
In particular, this also leads to a slight strengthening of Nickel's results on annihilation of Ext-groups and Fitting ideals derived in \cite[\S4]{NickelConductor}.
Finally, in \cref{sec:epac}, we will show that the equivariant $p$-adic Artin conjecture introduced in \cite{EpAC} holds for graduated orders under the assumption of the equivariant Iwasawa main conjecture.

\subsection{Notation and terminology}
The word \emph{ring} means not necessarily commutative ring with unity. All modules are unital. A \emph{domain} is a ring without zero divisors, an \emph{integral domain} is a commutative domain. {\emph{Indecomposable} idempotents are not necessarily central, and we reserve the term \emph{primitive} for indecomposable central idempotents.} We warn the reader that the word \emph{lattice} shall occur in both of its algebraic meanings (i.e. finitely generated torsionfree module vs. algebraic structure with meet and join): even though it should be clear from the context which one we mean, we will always use qualifiers such as $\LLambda$-lattice for the former and $(\cap,+)$-lattice for the latter. For $L$ a field and $\mathscr A$ a separable $L$-algebra, the notations $\Tr_{\mathscr A/L}$ and $\tr_{\mathscr A/L}$ stand for trace and reduced trace, respectively.

We will abuse notation by writing $\oplus$ for a direct product of rings, even though this is not a coproduct in the category of rings.
The Greek capital letters gamma and lambda are commonly used notation both in Iwasawa theory and in the theory of orders, but with different meanings; we shall use the usual $\Gamma$ and $\Lambda$ to denote Iwasawa-theoretic objects, and use blackboard letters $\GGamma$ and $\LLambda$ in order-theoretic contexts.

\subsection{Acknowledgements}
The author wishes to express his gratitude to Andreas Nickel for his careful reading of a draft version of this manuscript, and Justina Lückehe for helpful discussions. {He is grateful to the anonymous referee for their comments, which helped improve the clarity of exposition considerably.}

\section{Graduated orders in completed group rings} \label{sec:graduated}
In the {one}-dimensional case, Jacobinski's conductor formula for group rings holds for all maximal, and more generally, for all hereditary orders \cite[Satz~10.7]{JacobinskiEssen}. Working with group rings over {(complete) discrete valuation rings}, Plesken generalised this in two directions \cite[Theorem~III.8]{Plesken}: he replaced hereditary orders by graduated orders and the group ring by a {self-dual} order. 
The goal of this section is to adapt these results to more general settings, whenever possible, and to explicitly describe obstructions to such generalisations whenever this isn't the case.

\subsection{Inadequacy of hereditary orders over Iwasawa algebras} \label{sec:hereditary}
Recall that an order is called \emph{left hereditary} if all of its left ideals are projective.
For Dedekind domains (globally) and discrete valuation rings (locally), the notion of hereditary orders generalises that of maximal orders \cite[Theorems~17.3,~21.4]{MO}. Thereby hereditary orders offer a well-developed theory of a subclass of non-maximal orders, see e.g. \cite[Chapter~9]{MO}. However, while the definition of hereditary orders also makes sense over the Iwasawa algebra, it is no longer true that maximal orders are hereditary.

Let $p$ be an odd rational prime, $K/\QQ_p$ a finite extension, and $D$ a skew field of finite index $n$, centre $K$. Suppose that $n\mid p-1$.
Let $K/k$ be a finite cyclic Galois $p$-extension of degree $d$ and generator $\tau$. Then $\tau$ admits a unique extension to $D$ as an automorphism of order $d$ \cite[Proposition~2.7]{W}, which we will also denote by $\tau$.

Let $\OO_{D}$ be the unique maximal $\OO_K$-order in $D$ \cite[Theorem~12.8]{MO}. 
Let $\Omega\colonequals\OO_D[[X;\tau,\tau-1]]$ denote the skew power series ring: the underlying additive group is that of the power series ring $\OO_D[[X]]$, and the multiplicative structure is given by $Xd=\tau(d) X+\tau(d)-d$ for all $d\in\OO_D$; this gives rise to a well-defined ring structure.
Let $\mathscr D\colonequals \Quot(\Omega)$ be the total ring of quotients. The ring $\mathscr D$ is a skew field, $\Omega$ is a maximal order, and $\mathscr D$ has centre $\Frac(\OO_k[[T]])$, where $T=(1+X)^{d}-1$. See \cite[\S3]{W} for these statements.
This setup should be seen as the motivating example for our results: as we will explain in \cref{sec:maximal}, the skew fields $\mathscr D$ are those occurring in the Wedderburn decomposition of the algebras $\Q^F(\G)$, {where $\G$ is a one-dimensional admissible $p$-adic Lie group.}

The maximal order $\Omega$ in $\mathscr D$ is not right hereditary. Indeed, Cohn \cite[Theorem~4]{Cohn} showed that if $A$ is a right hereditary local ring such that $\cent(A)$ is not a field, then both $A$ and $\cent(A)$ are DVRs. (In Cohn's terminology, a DVR is a (non-commutative) domain with a prime element $P$ such that all nonzero elements are of the form $P^r u$ with $r\ge0$ and $u$ a unit.) The ring $A=\Omega$ is local \cite[Proposition~2.11]{VenjakobWPT}, and its centre $\cent(A)=\OO_{k}[[T]]$ is neither a field, nor a DVR by the Weierstraß preparation theorem. Hence $\Omega$ cannot be right hereditary.

Below we will show that this example is a special case of a more general setting in which hereditary orders don't exist: see \cref{no-hereditary-orders}.

\subsection{Graduated orders over two-dimensional rings} \label{sec:graduated-higher-dim}
The study of graduated orders goes back to works of Zassenhaus in the 1970s. Graduated orders form a generalisation of maximal orders (and of hereditary orders), and they have been extensively studied over (complete) discrete valuation rings. In the following, we will adapt some of these results to graduated orders over non-PIDs.

For the {one}-dimensional versions of the statements below, see the book \cite{Plesken} or the article \cite{ZassenhausPlesken}, but note that many proofs are omitted in these works. The habilitation thesis \cite{Plesken-habil} provides more proofs, but, being a thesis, it is of restricted availability. Supposedly, there is also an unpublished manuscript of Zassenhaus from 1975 which should contain proofs. In light of this state of affairs, when it comes to statements concerning graduated orders, we aim to keep this section as self-contained as possible, even when proofs in the two-dimensional case don't differ significantly from their {one}-dimensional counterparts.

{A remark is in order concerning the hypotheses present in these works. In the book \cite{Plesken}, the base ring is assumed to be a complete local valuation ring. The works \cite{ZassenhausPlesken} and \cite{Plesken-habil}, on the other hand, don't require completeness. Indeed, one only needs to be able to lift idempotents in the sense described in \cite[\S6c]{MO}: completeness is a sufficient condition for this property to hold. Furthermore, lifting idempotents is only necessary in some but not all of the proofs, so many statements hold without assuming completeness. Therefore completeness won't be part of the standing assumptions of this work. However, note that the base ring \emph{is} complete in the arithmetically important case described in \cref{sec:graduated}.}

Let $R$ be a noetherian integral domain, let $L=\Frac(R)$ be its field of fractions, and let $\mathscr A$ be a separable $L$-algebra.

\begin{definition} \label{def:graduated}
    An $R$-order $\LLambda\subset \mathscr A$ is called a \emph{graduated order} if there exist orthogonal indecomposable idempotents $\e_1,\ldots,\e_t\in \LLambda$ such that $\sum_{i=1}^t \e_i=1$ and $\e_i\LLambda\e_i\subset \e_i\mathscr A\e_i$ is a maximal order for all $i=1,\ldots,t$.
\end{definition}

Observe the following direct consequence of the definition:
\begin{lemma} \label{lem:contained-graduated-order}
    If $\LLambda\subseteq \GGamma\subset \mathscr A$ are two $R$-orders and $\LLambda$ is graduated, then so is $\GGamma$. \qed
\end{lemma}

{Since $\mathscr A$ is a separable $L$-algebra, it decomposes into simple components}
\[{\mathscr A=\bigoplus_{\chi\in \mathscr X}\mathscr A_\chi\simeq\bigoplus_{\chi\in \mathscr X} M_{n_\chi}(\mathscr D_\chi),}\]
{where $\chi$ runs over some finite indexing set $\mathscr X$, and $\mathscr D_\chi$ is a skew field whose centre $\cent(\mathscr D_\chi)$ is separable over $L$.}
{It is clear that $\LLambda$ is a graduated order in $\mathscr A$ if and only if $\LLambda\cap \mathscr A_\chi$ is a graduated order in $\mathscr A_\chi$ for each $\chi\in\mathscr X$. From now on, let $\mathscr A=M_n(\mathscr D)$, and drop the subscripts $\chi$.}

\begin{remark}
    {It is instructive to consider the following special case. Suppose that $L$ is a splitting field of $\mathscr A$, that is, if $\mathscr A=M_n(L)$. {If} $\LLambda\subset \mathscr A$ is an $R$-order containing orthogonal indecomposable idempotents $\e_1,\ldots,\e_t\in \LLambda$ such that $\sum_{i=1}^t \e_i=1$, then $\LLambda$ a is graduated order. Indeed, in this case, we have $\e_i\mathscr A\e_i=L$ for all $i=1,\ldots,n$, and the only $R$-order in $L$ is $R$ itself: in particular, $\e_i\LLambda\e_i$ is necessarily maximal. This is an adaptation of \cite[Remark~2.102]{Eisele} to our setting. Note that the definition of graduated orders in op.cit.~differs slightly from our \cref{def:graduated}: indeed, we require the idempotents $\mathsf e_i$ to be contained in $\LLambda$.}
\end{remark}

\subsubsection{Standard form} \label{sec:standard-form} \phantom-\medskip

\noindent
{Let $R$ be a regular local ring of dimension at most two. Let $\mathscr D$ be a skew field whose centre is separable over $L=\Frac (R)$, and let $\mathscr A=M_n(\mathscr D)$. Fix a maximal $R$-order $\Omega$ in $\mathscr D$.}

\begin{remark}
    Note that while maximal orders are unique in the {complete} one-dimensional case \cite[Theorem~12.8]{MO}, this is no longer true if $R$ is {complete} two-dimensional: in this case, all maximal orders are conjugate by \cite[Theorem~5.4]{Ramras2}, see also \cite[Remark~3.6]{W}. {If $R$ is not complete, then uniqueness fails even in the one-dimensional case.}
\end{remark}
\begin{remark} \label{rem:Omega-local}
    It follows from the assumptions that $\Omega$ is a local ring, that is, it has a unique two-sided maximal ideal $\mm_{\Omega}$. Indeed, if $R$ is one-dimensional, then {this is \cite[Theorem~18.3]{MO}}. In the two-dimensional case, the statement follows from \cite[Theorem~5.4]{Ramras2}.
\end{remark}

\begin{lemma} \label{lem:all-conjugates-maximal}
    {All} conjugates of $M_n(\Omega)$ by a unit in $\mathscr A^\times$ are maximal $R$-orders in $\mathscr A=M_n(\mathscr D)$. Conversely, every maximal $R$-order in $\mathscr A$ is conjugate to $M_n(\Omega)$.
\end{lemma}
\begin{proof}
    The first assertion holds more generally, see e.g.~\cite[Theorem~8.7]{MO}. The second assertion {is \cite[Theorem~18.7]{MO} for discrete valuation rings (recall that the one-dimensional regular local rings are precisely the discrete valuation rings). The two-dimensional statement follows by a direct application of Ramras's description of maximal orders over regular local rings of dimension two \cite[Theorem~5.4]{Ramras2}}.
\end{proof}

{In particular, the algebras considered in \cref{sec:hereditary} satisfy the hypotheses of \cref{lem:all-conjugates-maximal}. Therefore if $R=\Lambda^{\OO_F}(\Gamma_0)$ and $\mathscr A=\Q^F(\G)$, then the simple components satisfy these assumptions \cite[\S3 \& Theorem~4.9]{W}.}

\begin{remark}
    {It follows from \cref{lem:all-conjugates-maximal} that maximal {$R$-orders} are graduated. Indeed, by \cite[Theorem~10.5(i)]{MO}, it suffices to treat the case when $\mathscr A$ is simple. In the simple algebra $\mathscr A=M_n(\mathscr D)$, the maximal order $M_n(\Omega)$ is clearly graduated with respect to the idempotents $\mathsf f_1,\ldots,\mathsf f_n$, where $\mathsf f_\ell$ is the matrix with $1$ in the $\ell$th entry in the diagonal and zeros elsewhere, $1\le \ell\le n$. Therefore the maximal order $u M_n(\Omega) u^{-1}$ with $u\in \GL_n(\mathscr D)$ is graduated with respect to the idempotents $u\mathsf f_1 u^{-1},\ldots,u\mathsf f_n u^{-1}$.}
\end{remark}

\begin{definition} \label{def:standard-form}
    Let $t>0$ and $\mathbf{n}=(n_i)\in\ZZ_{>0}^t$ such that $\sum_{i=1}^t n_i=n$. Let $I_{ij}$ be two-sided nonzero fractional ideals of $\Omega$ for all $1\le i,j\le t$, such that
    \begin{definitionlist}
        \item \label{def:standard-form.i} $I_{ij} I_{jk}\subseteq I_{ik}$ for all $0\le i,j,k\le t$.
    \end{definitionlist}
    We write $\mathbf I=(I_{ij})$ for such a collection of fractional ideals.
    Let $\LLambda(\mathbf n, \mathbf I)$ denote the set of $n\times n$ matrices $A$ partitioned into $n_i\times n_j$ matrices $A_{ij}$ such that $A_{ij}\in M_{n_i\times n_j}(I_{ij})$.
    The set $\LLambda(\mathbf n, \mathbf I)\subset M_n(\mathscr D)$ is called a {\emph{graduated order in standard form}} if $\mathbf I$ consists of two-sided ideals and the following conditions hold:
    \begin{definitionlist}[resume*]
        \item \label{def:standard-form.ii} $I_{ii}=\Omega$ for all $0\le i\le t$,
        \item \label{def:standard-form.iii} $I_{ij} I_{ji} \subsetneq \Omega$ for all $0\le i\ne j\le t$.
    \end{definitionlist}
\end{definition}

\begin{remark}
    {\Cref{def:standard-form} generalises \cite[Definition~II.2]{Plesken}. In that work, non-negative integers $m_{ij}$ are used instead of ideals $I_{ij}$. The reason is that in the one-dimensional setting, every ideal is some power of the maximal ideal, and so it is sufficient to record the corresponding exponent, but this is no longer the case in the two-dimensional setting. Another difference is that we make a choice of a maximal order $\Omega$ {in $\mathscr D$}.}
\end{remark}

\begin{remark}
    Condition~(\ref{def:standard-form.i}) makes sure that the set $\LLambda(\mathbf n, \mathbf I)\subset M_n(\mathscr D)$ is a subring, and being a graduated order follows from (\ref{def:standard-form.ii}). Condition~(\ref{def:standard-form.iii}) helps avoid trivial subdivisions: for instance, otherwise the graduated order $M_2(\Omega)$ would have standard form with $t=1$, $n_1=2$, $I_{11}=\Omega$ on the one hand, and $t=2$, $n_i=1$, $I_{ij}=\Omega$ for $i,j=1,2$ on the other.
\end{remark}

\begin{definition}
    Let $\LLambda\subset\mathscr A$ be an $R$-order. Recall that two idempotents $\mathsf e_i,\mathsf e_j \in\mathscr A$ are said to be \emph{associates in $\LLambda$} if there exits $x,y\in\LLambda$ such that $x\in \mathsf e_i\LLambda \mathsf e_j$ and $y\in \mathsf e_j\LLambda\mathsf e_i$ such that $\mathsf e_i=xy$ and $\mathsf e_j=yx$. We write $\mathsf e_i\sim\mathsf e_j$.
\end{definition}

\begin{proposition} \label{every-standard-form}
    Every graduated $R$-order in $\mathscr A=M_n(\mathscr D)$ is isomorphic to one in standard form.
\end{proposition}
See \cite[Remark~II.3]{Plesken} resp. \cite[Satz~I.4.a]{Plesken-habil} for the statement resp. proof in the {one}-dimensional version.
\begin{proof}
    The proof models that in loc.cit.
    If $\LLambda$ is a graduated order with respect to the idempotents $\mathsf f_\ell\in \mathscr A$, where $\mathsf f_\ell$ is the matrix with $1$ in the $\ell$th entry in the diagonal and zeros elsewhere, then $\LLambda$ is as in \cref{def:standard-form}. 
    
    Now suppose that $\LLambda$ is a graduated order with respect to some orthogonal indecomposable idempotents $\e_1,\ldots,\e_{n'}\in \LLambda$. 
    Any two full sets of orthogonal idempotents can be conjugated into each other by a matrix in $\mathscr A^\times$ \cite[\S6, Exercise~14]{CR}: this forces $n=n'$, and we may assume $\e_\ell=\mathsf f_\ell$ for all $1\le \ell\le n$. 
    
    Moreover, we may assume $\LLambda\subseteq M_n(\Omega)$. For this, first note that $\LLambda \mathsf f_k$ is a faithful $\LLambda$-module, that is, for all $A\ne B\in \LLambda$ there is an $M\in \LLambda$ such that $A\cdot M\mathsf f_k\ne B\cdot M\mathsf f_k$. Indeed, if $A=(a_{k\ell})$ and $B=(b_{k\ell})$, and $a_{k_0\ell_0}\ne b_{k_0\ell_0}$ for some $1\le k_0,\ell_0\le n$, then let $M$ be a matrix whose single non-zero entry $m\ne 0$ is in the position $(\ell_0,k)$. Then the $(k_0,k)$-entries of $A M\mathsf f_k$ resp. $BM\mathsf f_k$ are $a_{k_0\ell_0}m\ne b_{k_0\ell_0}m$. So left multiplication gives rise to an embedding 
    \[{\LLambda\hookrightarrow\End_{\Omega}(\LLambda\mathsf f_k)\hookrightarrow \End_{\Omega}(\Omega^n)\simeq M_n(\Omega).}\] 
    The second arrow here comes from the embedding $\LLambda\mathsf f_k\hookrightarrow\Omega^n$ given by clearing denominators, that is, multiplying by a diagonal matrix. Therefore after conjugating with a diagonal matrix, we have $\LLambda\subseteq M_n(\Omega)$; note that conjugation by a diagonal matrix acts trivially on the idempotents $\mathsf f_k$.

    Reordering rows and columns -- that is, conjugating by a permutation matrix --, we may assume that $\mathsf f_1\sim\ldots\sim\mathsf f_{n_1}$, $\mathsf f_{n_1+1}\sim\ldots\sim\mathsf f_{n_2}$, \ldots, $\mathsf f_{n_{t-1}+1}\sim\ldots\sim\mathsf f_{n_t}$, where $\sum_{i=1}^t n_i=n$. 
    We have $\mathsf f_k \LLambda\mathsf f_\ell\subseteq \mathsf f_k M_n(\Omega) \mathsf f_\ell =\Omega$, so $\mathsf f_k\LLambda\mathsf f_\ell \equalscolon \tilde I_{k\ell}$ is an integral ideal in $\Omega$. It follows from the definition of associates that $\mathsf f_k\sim\mathsf f_\ell$ iff $\tilde I_{k\ell} \tilde I_{\ell k}=\Omega$, that is, iff $\tilde I_{k\ell}=\tilde I_{\ell k}=\Omega$.
    Therefore for all $1\le i\le t$, we have
    \[\left(\sum_{k=n_{i-1}+1}^{n_i} \mathsf f_k\right) \LLambda \left(\sum_{k=n_{i-1}+1}^{n_i} \mathsf f_k\right) = M_{n_i}(\Omega),\]
    where $n_0$ is understood to be $0$. Furthermore, for all $1\le i,j\le t$,
    \[\left(\sum_{k=n_{i-1}+1}^{n_i} \mathsf f_k\right) \LLambda \left(\sum_{\ell=n_{j-1}+1}^{n_j} \mathsf f_\ell\right) = M_{n_i\times n_j}(I_{ij}),\]
    where $I_{ij}=\tilde I_{n_in_j}$. If $i\ne j$, then $\mathsf f_i\not\sim\mathsf f_j$ implies $I_{ij}I_{ji}\subsetneq\Omega$.
\end{proof}

\begin{remark} \label{rem:graduated-wrt-f}
    Note that the graduated $R$-orders in $\mathscr A=M_n(\mathscr D)$ in standard form are precisely those that are graduated with respect to the indecomposable idempotents $\mathsf f_1,\ldots,\mathsf f_n$.
\end{remark}

\begin{corollary} \label{central-conductor-independence}
    Let $\GGamma\subset\mathscr A$ be a graduated $R$-order. Then $\cent(\GGamma)=\cent(\Omega)$, {so the centre of a graduated order is a maximal $R$-order in $\cent(\mathscr A)$}. In particular, if $\LLambda\subseteq \GGamma\subseteq \mathfrak M\subset \mathscr A$ are $R$-orders with $\mathfrak M$ a maximal order, then $\mathcal F(\GGamma/\LLambda)=\mathcal F(\mathfrak M/\LLambda)$.
\end{corollary}
\begin{proof}
    By \cref{every-standard-form}, we may assume that $\GGamma$ is of standard form. {Then $\cent(\GGamma)=\cent(\Omega)\cdot \mathbf 1_n$ and $\cent(\mathscr A)=\cent(\mathscr D)\cdot \mathbf 1_n$, where $\mathbf 1_n$ is the $n\times n$ identity matrix. Since $R\subseteq\cent(\mathscr D)$, and since $\Omega\subseteq \mathscr D$ is an $R$-order, we have that $\cent(\Omega)\subseteq \cent(\mathscr D)$ is a maximal $R$-order. This proves the first assertion.} Independence of the central conductor of $\GGamma$ then follows from its definition \eqref{eq:def-central-conductor}.
\end{proof}

As we shall see, the existence of non-principal ideals in $\Omega$ leads to weaker statements than in the {one}-dimensional case. The following statement is a partial solution to this problem.
\begin{lemma} \label{lem:principalisation}
    Let $\LLambda(\mathbf n,\mathbf I)\subset\mathscr A$ be a graduated $R$-order in standard form. Then $\LLambda(\mathbf n,\mathbf I)$ contains a graduated $R$-order $\LLambda(\mathbf n,\mathbf I')$ in standard form whose ideal matrix $\mathbf I'$ consists of principal ideals only.
\end{lemma}
\begin{proof}
    If $\mathbf n=n_1$, there is nothing to prove. Otherwise, let
    \[J \colonequals \bigcap_{1\le i,j\le t} I_{ij}\]
    be the intersection of the ideals in $\mathbf I$: since $\Omega$ has no zero divisors, this is a proper nonzero ideal of $\Omega$. Let $x\in J-\{0\}$.
    Define $\mathbf I'$ by setting
    \[
        I'_{ii}\colonequals \begin{cases}
            \Omega & \text{if $i=j$} \\
            x\Omega & \text{if $i\ne j$}
        \end{cases}
    \]
    Then $\LLambda(\mathbf n, \mathbf I')\subseteq \LLambda(\mathbf n,\mathbf I)$ fulfills the conditions in \cref{def:standard-form}.
\end{proof}

\begin{notation}
    We introduce the following notation.
    \begin{definitionlist}
    \item Let $\mathbf I=(I_{ij})_{1\le i,j\le n}$ and $\mathbf I'=(I'_{ij})_{1\le i,j\le n}$ be two collections of nonzero fractional ideals of $\Omega$.\footnote{For sake of exposition, we will use the terminology of matrices for these objects, although they aren't matrices themselves.}
    Let $\mathbf I+\mathbf I'\colonequals(I_{ij}+I'_{ij})$ resp. $\mathbf I*\mathbf I'\colonequals(I_{ij}I'_{ij})$ denote their entry-wise sum resp. product. 
    \item Let $\mathbf I_j=(I_{ij})_{1\le i\le n}$ resp. ${}_i\mathbf I=(I^{-1}_{ij})_{1\le j\le n}$ denote the $j$th column of $\mathbf I$ resp. the $i$th column of $\mathbf I^{-,\top}$, where $\mathbf I^{-,\top}$ is the transpose of the entry-wise inverse of $\mathbf I$, provided that all fractional ideals $I_{ij}$ are invertible.
    \item If $I$ is a fractional ideal of $\Omega$, let $I\mathbf 1$ denote the $n\times n$-matrix with $I$ in its diagonal and the zero ideal elsewhere. 
    \item Let $\mathbf E \in M_{n\times n}(\mathscr D)$ be the $n\times n$ matrix consisting entirely of ones. Let $\mathbf E_{\mathfrak m_\Omega}\colonequals \mathfrak m_\Omega \mathbf 1+\Omega(\mathbf E-\mathbf 1)$, so that $\mathbf E_{\mathfrak m_\Omega,ii}=\mathfrak m_\Omega$ and $\mathbf E_{\mathfrak m_\Omega,ij}=\Omega$ for $i\ne j$, {where $\mm_\Omega$ is the unique maximal ideal in $\Omega$ (see \cref{rem:Omega-local}).}
    \item For a vector $\mathbf m=(m_i)$ of length $t$ of fractional ideals of $\Omega$, let $\mathbb L(\mathbf m)\le \mathscr D^n$ be the set of vectors of length $n$ with the first $n_1$ entries in $m_1$, the next $n_2$ entries in $m_2$ and so forth. 
    \item Let $\mathfrak Z(\mathscr D^n)$ be the set of nonzero left $\LLambda(\mathbf n,\mathbf I)$-lattices in $\mathscr D^n$, where the $\LLambda(\mathbf n,\mathbf I)$-module structure on $\mathscr D^n$ is by left multiplication.
    \end{definitionlist}
\end{notation}

The following is an adaptation of \cite[Remark~II.4]{Plesken} to our setting; see also \cite[Satz~I.4.b]{Plesken-habil}.
\begin{lemma}
    Let $\mathbf n$ and $\mathbf I$ be as {in \cref{def:standard-form}}. Then the following hold.
    \begin{lemmalist}
    \item \label{II.4.i} The Jacobson radical of $\LLambda(\mathbf n,\mathbf I)$ is $\Jac(\LLambda(\mathbf n,\mathbf I))=\LLambda(\mathbf n, \mathbf I*\mathbf E_{\mathfrak m_\Omega})$.
    \item \label{II.4.ii} The quotient is the semisimple algebra $\LLambda(\mathbf n,\mathbf I)/\Jac(\LLambda(\mathbf n,\mathbf I))\simeq \bigoplus_{i=1}^t M_{n_i}(\Omega/\mm_\Omega)$.
    \end{lemmalist}
\end{lemma}
\begin{proof}
    It is clear that the quotient $\LLambda(\mathbf n,\mathbf I)/\LLambda(\mathbf n, \mathbf I*\mathbf E_{\mathfrak m_\Omega})\simeq \bigoplus_{i=1}^t M_{n_i}(\Omega/\mm_\Omega)$ is semisimple. Since the Jacobson radical is the smallest ideal with semisimple quotient, it follows that $\Jac(\LLambda(\mathbf n,\mathbf I))\subseteq \LLambda(\mathbf n, \mathbf I*\mathbf E_{\mathfrak m_\Omega})$. 
    
    We claim that every matrix in $\LLambda(\mathbf n, \mathbf I*\mathbf E_{\mathfrak m_\Omega})$ with a single nonzero entry is in the Jacobson radical. Since $\Jac(\LLambda(\mathbf n,\mathbf I))$ is an ideal and since every matrix in $\LLambda(\mathbf n, \mathbf I*\mathbf E_{\mathfrak m_\Omega})$ is a linear combination of such matrices, this will show that $\Jac(\LLambda(\mathbf n,\mathbf I))\supseteq \LLambda(\mathbf n, \mathbf I*\mathbf E_{\mathfrak m_\Omega})$.
    
    Let $1\le k,\ell\le n$, and let $E_{k\ell}$ denote the $n\times n$ matrix with $1$ in the $(k,\ell)$-entry and zeros elsewhere. Let $i$ and $j$ be the corresponding row and column in $\mathbf I$, that is, $\sum_{r=1}^{i-1} n_r < k \le \sum_{r=1}^i n_r$ and $\sum_{r=1}^{j-1} n_r< \ell \le \sum_{r=1}^j n_r$. Let $x\in I_{ij}$ if $i\ne j$, and $x\in I_{ii}\mm_\Omega=\mm_\Omega$ if $i=j$. Consider the matrix $x E_{k\ell}$.
    
    Let $A\in \LLambda(\mathbf n,\mathbf I)$ be arbitrary: then $\mathbf 1-AxE_{k\ell}$ is invertible. Indeed, if $i=j$, then this is clear from $x\in \mm_\Omega$, and if $i\ne j$, then this follows from \cref{def:standard-form.iii}. Since $\mathbf 1-AxE_{k\ell}$ is invertible for every $A$, we have $x E_{k\ell}\in\Jac(\LLambda(\mathbf n,\mathbf I))$, as claimed. This proves (\ref{II.4.i}), and (\ref{II.4.ii}) follows as a direct consequence.
\end{proof}

\subsubsection{$\LLambda$-lattices} \label{sec:lattices} \phantom-\medskip

\noindent

\begin{definition}
    Let {$R$ be a noetherian integral domain and} $\LLambda$ an $R$-order. Then a \emph{$\LLambda$-lattice} is a $\LLambda$-module that is an $R$-lattice, that is, a $\LLambda$-module that is finitely generated over $R$ and torsion-free over $R$.
\end{definition}

Nonzero $\LLambda$-lattices are described as follows. See \cite[Remark~II.4]{Plesken} and \cite[Theorem~3.18]{ZassenhausPlesken} for the {one}-dimensional versions.

\begin{lemma} \label{II.4}
    {Suppose that $R$, $\mathscr D$, and $\Omega$ are as in \cref{sec:standard-form}, and let $\mathbf n$ and $\mathbf I$ be as in \cref{def:standard-form}}. Then the following hold.
    \begin{lemmalist}
        \item \label{II.4.iii} The set of nonzero left $\LLambda(\mathbf n,\mathbf I)$-lattices in $\mathscr D^n$ is $\mathfrak Z(\mathscr D^n)=\{\mathbb L(\mathbf m): \forall i,j: I_{ij} m_j\subseteq m_i\}$.
        \item \label{II.4.viii} The two-sided fractional ideals of $\LLambda(\mathbf n,\mathbf I)$ are of the form $\LLambda(\mathbf n,\mathbf J)$ with $\mathbf J$ satisfying $I_{ij} J_{jk} \subseteq J_{ik}$ and $J_{ij} I_{jk}\subseteq J_{ik}$ for all $1\le i,j,k\le t$.
    \end{lemmalist}
\end{lemma}

\begin{proof}
    (\ref{II.4.iii}) On the one hand, it is clear that every such $\mathbb L(\mathbf m)$ is in $\mathfrak Z(\mathscr D^n)$. On the other hand, let $L\in\mathfrak Z(\mathscr D^n)$. Considering that $\LLambda(\mathbf n,\mathbf I) L=L$ and expanding the product on the left hand side, one sees that each coordinate in $L$ runs through a fractional ideal of $\Omega$. By the shape of the matrices in $\LLambda(\mathbf n,\mathbf I)$, it follows that the first $n_1$, then the following $n_2$, and so forth, of these fractional ideals must agree -- so $L=\mathbb L(\mathbf m)$ for some $\mathbf m$. Once again using the assumption that $L$ is a left $\LLambda(\mathbf n,\mathbf I)$-module, the condition $I_{ij} m_j\subseteq m_i$ for all $i,j$ follows.

    (\ref{II.4.viii}) It is clear that the sets $\LLambda(\mathbf n,\mathbf J)$ are fractional ideals.
    A fractional ideal of $\LLambda(\mathbf n,\mathbf I)$ is a full left and right $\LLambda(\mathbf n,\mathbf I)$-lattice in $M_n(\mathscr D)$, hence the converse follows from (\ref{II.4.iii}).
    Indeed, (\ref{II.4.iii}) shows that its columns must be of the form $\mathbb L(\mathbf m^{(i)})$, such that the collection $\mathbf J=(\mathbf m^{(i)})$ with columns $\mathbf m^{(i)}$ satisfies the first condition in the statement. The analogue of (\ref{II.4.iii}) for right $\LLambda(\mathbf n,\mathbf I)$-lattices, the proof of which goes along the same lines as that of the left version, shows necessity of the second condition on $\mathbf J$.
\end{proof}

In the {one}-dimensional case, the counterparts of the following observations serve as the basis of the study of left $\LLambda(\mathbf n,\mathbf I)$-lattices \cite[Sätze~I.8 \& I.23]{Plesken-habil}, \cite[Remark~II.4(iv--vii)]{Plesken}.

\begin{lemma}
    In addition to the assumptions {of \cref{II.4}}, suppose that $R$ is a complete local ring.
    \begin{propositionlist}
    \item \label{II.4.iv} There is an isomorphism $\mathbb L(\mathbf m)\simeq \mathbb L(\mathbf m')$ of $\LLambda(\mathbf n,\mathbf I)$-lattices if any only if $\mathbf m=\alpha \mathbf m'$ for some $\alpha\in\mathscr D^\times$.
    \item \label{II.4.v} The set $\mathbb L(\mathbf I_i)$ is a projective and indecomposable left $\LLambda(\mathbf n,\mathbf I)$-lattice. Conversely, every projective indecomposable left $\LLambda(\mathbf n,\mathbf I)$-lattice is isomorphic to $\mathbb L(\mathbf I_i)$ for some $i$.
    \item \label{II.4.vii} The set $\mathbb L({}_i\mathbf I)$ is a projective and indecomposable right $\LLambda(\mathbf n,\mathbf I)$-lattice. Conversely, every projective indecomposable right $\LLambda(\mathbf n,\mathbf I)$-lattice is isomorphic to $\mathbb L({}_i\mathbf I)$ for some $i$.
    \end{propositionlist}
\end{lemma}
\begin{proof}
    Recall the following general fact, see e.g. \cite[368]{Facchini}: if $\mathfrak R$ is a domain and $\mathfrak a,\mathfrak b\subseteq \mathfrak R$ are two-sided ideals, then there is an isomorphism $f:\mathfrak a\to\mathfrak b$ of left $\mathfrak R$-modules if and only if there is a unit $\alpha\in\Frac(\mathfrak R)^\times$ such that $\mathfrak a=\alpha \mathfrak b$. Indeed, the claim is trivial if $\mathfrak a=(0)$, and if $\mathfrak a\ne(0)$, then for any $a\in\mathfrak a-\{0\}$, one has $f(a)\mathfrak a=\{f(a)x:x\in \mathfrak a\}=\{f(ax):x\in\mathfrak a\}=\{a f(x):x\in\mathfrak a\}=\{ay:y\in \mathfrak b\}=a\mathfrak b$, so $\alpha\colonequals f(a)^{-1}a$ works. Note that $\alpha$ is determined at most up to a unit in $\mathfrak R^\times$.
    
    (\ref{II.4.iv}) The `only if' part is clear. Conversely, let a $\LLambda(\mathbf n,\mathbf I)$-module isomorphism $f:\mathbb L(\mathbf m)\xrightarrow{\sim}\mathbb L(\mathbf m')$ be given. For $1\le i\le t$, let $A^{(i)}\colonequals\sum_{k=1}^{n_i} \mathsf f_{\sum_{r=1}^{i-1} n_r + k}\in M_{n\times n}(\mathscr D)$ be the diagonal matrix with ones along the diagonal of the $i$th block and zeros elsewhere.
    For all $i$, we have $f(A^{(i)} \mathbb L(\mathbf m))=A^{(i)}\mathbb L(\mathbf m')$, so the fact above provides $\alpha_i\in\mathscr D^\times$ such that $m_{i}=\alpha_i m'_{i}$ for all $1\le i\le t$. Since $f$ was a $\LLambda(\mathbf n,\mathbf I)$-module isomorphism, so is left multiplication by the diagonal matrix $A\colonequals\sum_{i=1}^{t} \sum_{s=1}^{n_i} \alpha_i \mathsf f_{\sum_{r=1}^{i-1} n_r + s}$ with the $\alpha_i$s in its diagonal. Hence $A\in\cent(\LLambda(\mathbf n,\mathbf I))$, so $A$ is a scalar matrix, and all $\alpha_i$s can be chosen to be equal.

    (\ref{II.4.v}) By definition, $\bigoplus_{i=1}^t \mathbb L(\mathbf I_i)^{\oplus n_i}=\LLambda(\mathbf n, \mathbf I)$, so each $\mathbb L(\mathbf I_i)$ is a direct summand of a free $\LLambda(\mathbf n,\mathbf I)$-module and thus projective. We have $\End_{\LLambda(\mathbf n,\mathbf I_i)} \mathbb L(\mathbf I_i)=\Omega$ is a local ring, hence $\mathbb L(\mathbf I_i)$ is indecomposable \cite[Proposition~30.5]{CR}; {this is where the completeness of $R$ is used}. The converse statement follows by an application of \cref{II.4.iv} and the Krull--Schmidt--Azumaya theorem \cite[Theorem~30.6]{CR}.

    (\ref{II.4.vii}) A $\LLambda(\mathbf n,\mathbf I)$-lattice is a finitely generated module, hence it is projective iff its $R$-dual is projective, so (\ref{II.4.vii}) is equivalent to (\ref{II.4.v}).
\end{proof}

\subsubsection{Inverse differents} \label{sec:inverse-different}

\begin{definition} \label{def:D-lattices}
    Let {$R$ be a noetherian integrally closed integral domain with field of fractions $L$, let} $\mathscr A$ be a separable $L$-algebra, let $\LLambda\subset\mathscr A$ be an $R$-order, and let $M\subset \mathscr A$ be a full left $\LLambda$-lattice. The associated dual right $\LLambda$-lattices with respect to the (reduced) trace are defined as follows:
    \begin{align*}
        \D(M/R)=\D_{\mathscr A/L}(M/R) &\colonequals \{a\in\mathscr A: \tr_{\mathscr A/L}(a M) \subseteq R \}, \\
        \D_{\ord}(M/R) &\colonequals \{a\in\mathscr A: \Tr_{\mathscr A/L}(a M) \subseteq R \}.
    \end{align*}
    Here the reduced trace is the composition $\tr_{\mathscr A/L}=\tr_{\cent(\mathscr A)/L}\circ\tr_{\mathscr A/\cent(\mathscr A)}$.
    If there is no risk of confusion, we shall suppress the algebra $\mathscr A$ and the field $L$ from the notation.
    The inverse different of $M$ over $R$ is $\D(M/R)$. The different is defined as $\DD(M/R)\colonequals\D(M/R)^{-1}$ whenever $\D(M/R)$ is invertible.
\end{definition}

The behaviour of inverse differents under a change of the ring $R$ is described by the following:
\begin{lemma}[{\cite[Lemma~2.1]{NickelConductor}}] \label{N2.1}
    {Let $R$, $L$, $\mathscr A$, and $M$ be as in \cref{def:D-lattices}, and suppose that $\mathscr A$ is a simple $L$-algebra.}
    Let $L'$ be an intermediate field in the extension $\cent(\mathscr A)/L$, and let $L'$ be the integral closure of $R$ in $L'$. Then
    \[\D(M/R)\supseteq \D(M/R')\D(R'/R),\]
    and this containment becomes an equality whenever $\D(R'/R)$ is invertible.
\end{lemma}

\begin{proposition} \label{different-graduated}
    {Suppose that $R$, $\mathscr D$, and $\Omega$ are as in \cref{sec:standard-form}, and let $\mathbf n$ and $\mathbf I$ be as in \cref{def:standard-form}.}
    The inverse different of a graduated order in standard form is
    \[\D(\LLambda(\mathbf n, \mathbf I)/R) = \LLambda\left(\mathbf n, \D(\Omega/R)\mathbf E * \mathbf I^{-,\top}\right),\]
    provided that all fractional ideals $I_{ij}$ are invertible.
\end{proposition}
\begin{remark} \label{rem:invariants-of-grad-orders}
    A fractional ideal is invertible if any only if it is projective; in particular, this is always the case in hereditary rings such as Dedekind domains. The maximal ideal $(p,T)\subset\ZZ_p[[T]]$ is, on the other hand, not invertible. The existence of non-invertible ideals also presents an obstacle to the generalisation of Zassenhaus's structural invariants described in \cite[p.~6ff.]{Plesken-habil}, \cite[p.~10ff.]{Plesken} for {(complete) discrete valuation rings}. These invariants can be used to describe graduated orders uniquely -- as opposed to the standard form, which isn't unique, not even up to permutation \cite[6]{Plesken-habil}.
    Not aiming for an extensive study of the $(\cap,+)$-lattice of $\LLambda$-lattices, we can avoid the need to construct such a set of invariants.
\end{remark}

\begin{proof}[First proof of \cref{different-graduated}]
    The reduced trace depends only on the diagonal entries, so we have $\tr_{\mathscr A/L}(\LLambda(\mathbf n, \D(\Omega/R)\mathbf E * \mathbf I^{-,\top}))\subseteq R$. 
    The inverse different $\D(\LLambda(\mathbf n, \mathbf I)/R)$ is the largest fractional ideal of $\LLambda(\mathbf n, \mathbf I)$ such that 
    $\tr_{\mathscr A/L}(\D(\LLambda(\mathbf n, \mathbf I)/R))\subseteq R$,
    so it follows that 
    \[{\LLambda\left(\mathbf n, \D(\Omega/R)\mathbf E * \mathbf I^{-,\top}\right) \subseteq \D(\LLambda(\mathbf n, \mathbf I)/R).}\]
    
    By \cref{II.4.viii}, we have $\D(\LLambda(\mathbf n, \mathbf I)/R)=\LLambda(\mathbf n, \mathbf J)$ for some $\mathbf J=(J_{ij})$ satisfying conditions \ref{def:standard-form.i} and \ref{II.4.viii}. Therefore $J_{ii}$ is the largest fractional ideal of $\Omega$ such that $\tr_{\mathscr D/L}(J_{ii})\subseteq R$, and so $J_{ii}=\D(\Omega/R)$. Indeed, in the previous paragraph we have seen that $\D(\Omega/R)\subseteq J_{ii}$, and if there were some $x\in J_{ii}-\D(\Omega/R)$, then $\tr_{\mathscr A/L}(x \mathsf f_k)=\tr_{\mathscr D/L}(x)\notin R$ where $\sum_{r=1}^{i-1} n_r < k \le \sum_{r=1}^i n_r$.
    
    Let $\LLambda(\mathbf m,\mathbf K)$ be a fractional ideal of $\LLambda(\mathbf n,\mathbf I)$ such that $\LLambda(\mathbf n, \D(\Omega/R)\mathbf E * \mathbf I^{-,\top})\subsetneq \LLambda(\mathbf m,\mathbf K)$. Recall from \cref{II.4.viii} that all fractional ideals are of this form, and by taking a common refinement of $\mathbf m$ and $\mathbf n$, we may assume them to be equal.
    Since $(\D(\Omega/R)\mathbf E * \mathbf I^{-,\top})_{ii}=J_{ii}$, it follows that $J_{ii}\subseteq K_{ii}$ for all $i$.
    
    By definition, the left $\LLambda(\mathbf n,\mathbf I)$-lattice $\LLambda(\mathbf n, \D(\Omega/R)\mathbf E * \mathbf I^{-,\top})$ can be written as a direct sum:
    \[\LLambda\left(\mathbf n, \D(\Omega/R)\mathbf E * \mathbf I^{-,\top}\right) = \bigoplus_{i=1}^t \mathbb L\left(\D(\Omega/R)\mathbf E_i * {}_i\mathbf I\right)^{\oplus n_i}.\]
    The fractional ideal $\LLambda(\mathbf m,\mathbf K)$ can also be decomposed like this: it is the direct sum of indecomposable $\LLambda(\mathbf n,\mathbf I)$-lattices $L_i$ such that for all indices $1\le i\le t$, we have $L_i \supseteq \mathbb L(\D(\Omega/R)\mathbf E_i * {}_i\mathbf I)$, and there exists an index $i'$ for which $L_{i'}\supsetneq \mathbb L(\D(\Omega/R)\mathbf E_{i'} * {}_{i'}\mathbf I)$. Write $L_{i'}=\mathbb L(\mathbf k)$ for some $\mathbf k$ as in \cref{II.4.iii}.
    
    Since $L_{i'}$ is an proper $\LLambda(\mathbf n,\mathbf I)$-overlattice, we have $\mathbf k_j\supseteq \D(\Omega/R) \mathbf I_{ji'}^{-1}$ for all $1\le j\le t$, and $\mathbf k_{j'}\supsetneq \D(\Omega/R) \mathbf I_{j'i'}^{-1}$ for at least one $j'$. From the condition in \cref{II.4.iii}, it follows that $j'=i'$. Indeed, if $j'\ne i'$, then the condition $I_{i'j'} k_{j'} \subseteq k_{i'}=\D(\Omega/R)$ fails.

    Returning to the fractional ideal $\LLambda(\mathbf m,\mathbf K)$, this means that $K_{i'i'}\supsetneq J_{ii}$; in particular, $\LLambda(\mathbf m,\mathbf K)\not\subseteq \D(\LLambda(\mathbf n, \mathbf I)/R)=\LLambda(\mathbf n,\mathbf J)$.
    So the containment $\LLambda(\mathbf n, \D(\Omega/R)\mathbf E * \mathbf I^{-,\top}) \subseteq \D(\LLambda(\mathbf n, \mathbf I)/R)$ is an equality, as claimed.
\end{proof}

The author {would like to thank} Andreas Nickel for pointing out that \cref{different-graduated} can also be proven directly from the definitions:

\begin{proof}[Second proof of \cref{different-graduated}]
We first show the inclusion $\D(\LLambda(\mathbf n, \mathbf I)/R) \supseteq \LLambda\left(\mathbf n, \D(\Omega/R)\mathbf E * \mathbf I^{-,\top}\right)$. This is equivalent to verifying the containment
\[\tr_{\mathscr A/L}\left(\LLambda(\mathbf n, \mathbf I) \cdot \LLambda\left(\mathbf n, \D(\Omega/R)\mathbf E * \mathbf I^{-,\top}\right)\right) \subseteq R.\]
Let $A\in \LLambda(\mathbf n, \mathbf I)$ and $B\in \LLambda\left(\mathbf n, \D(\Omega/R)\mathbf E * \mathbf I^{-,\top}\right)$, and write $AB=(c_{k\ell})_{1\le k,\ell\le n}$. Then $\tr_{\mathscr A/L}(AB)=\sum_{k=1}^n \tr_{\mathscr D/L}(c_{kk})$. As $c_{kk}$ is a diagonal entry of $AB$, it is the sum of elements in $I_{ij} \D(\Omega/R) I^{-1}_{ij} = \D(\Omega/R)$ where $i$ is such that $\sum_{r=1}^{i-1} n_r < k \le \sum_{r=1}^i n_r$ and $1\le j\le t$. Hence $c_{kk}\in\D(\Omega/R)$, and thus $\tr_{\mathscr A/L}(AB)\in R$, as desired.

For the reverse inclusion, let $A=(a_{k\ell}\in \D(\LLambda(\mathbf n,\mathbf I)/R)$. We show that $a_{k\ell}\in \D(\Omega/R)\cdot I_{ji}^{-1}$, where $\sum_{r=1}^{i-1} n_r < k \le \sum_{r=1}^i n_r$ and $\sum_{r=1}^{j-1} n_r< \ell \le \sum_{r=1}^j n_r$, that is, $a_{k\ell}$ is in the $(i,j)$-block. Let $P_{k\ell}$ be a transposition matrix such that $AP_{k,\ell}$ is the matrix $A$ with its $k$th and $\ell$th columns swapped, so that the $k$th diagonal entry of $AP_{k,\ell}$ is $a_{k\ell}$. Let $N_k$ be a matrix whose single nonzero entry is $1$, in the $k$th diagonal entry, so that $AP_{k\ell}N_k$ has $k$th column identical to that of $AP_{k\ell}$ and zeros elsewhere. Let $\omega \in I_{ij}$: then $\tr_{\mathscr A/L}(AP_{k,\ell}N_k\omega)=\tr_{\mathscr D/L}(a_{k\ell}\omega)\in R$ by the assumption on $A$. Hence $a_{k\ell}\in \D(\Omega/R)\cdot I_{ji}^{-1}$, as claimed.
\end{proof}

\subsubsection{Intersections of maximal orders} \phantom-\medskip

\noindent
In the {one}-dimensional case, one can show that graduated orders are essentially intersections of maximal orders; this uses principality and invertability of ideals of $\Omega$ \cite[Theorem~II.8]{Plesken}. In dimension $2$, a slightly weaker statement is true.

\begin{example} \label{ex:intersections}
    Consider the setup in \cref{sec:hereditary}, and let $\mathbf n=(1,1)$ for simplicity. Then $M_2(\Omega)\subset M_2(\mathscr D)$ is a maximal order. By \cref{lem:all-conjugates-maximal}, the maximal orders in $M_2(\mathscr D)$ are precisely the conjugates of this order. In particular, for all $0\ne d\in \Omega$, the following is a graduated order in standard form:
    \begin{equation} \label{eq:intersection-principal}
        \begin{pmatrix} d & \\ & 1 \end{pmatrix} \begin{pmatrix} \Omega & \Omega \\ \Omega & \Omega \end{pmatrix}\begin{pmatrix} d^{-1} & \\ & 1 \end{pmatrix} \cap \begin{pmatrix} \Omega & \Omega \\ \Omega & \Omega \end{pmatrix} = \begin{pmatrix} \Omega & d\Omega \\ d^{-1}\Omega & \Omega \end{pmatrix} \cap \begin{pmatrix} \Omega & \Omega \\ \Omega & \Omega \end{pmatrix} = \begin{pmatrix} \Omega & d\Omega \\ \Omega & \Omega \end{pmatrix}.
    \end{equation}
\end{example}
\begin{example}
    Not all graduated orders arise as an intersection of maximal orders. Consider the case $k=K=D$, so that $\Omega=\OO_K[[T]]$ is commutative.
    We claim that the graduated order $\LLambda=\begin{psmallmatrix} \Omega & \mm_\Omega \\ \Omega & \Omega \end{psmallmatrix}$ is not an intersection of maximal orders. 
    
    Suppose it is: then it is contained in a maximal order $A M_2(\Omega) A^{-1}$ for some $A=\begin{psmallmatrix} a&b\\c&d \end{psmallmatrix}\in M_2(\Omega)$ with $ad-bc\ne0$. Let $\pi\in \OO_K$ be a uniformiser. By the Weierstraß preparation theorem \cite[Theorem~5.3.4]{NSW}, each entry $*\in\{a,b,c,d\}$ of $A$ is of the form $\pi^{e_*} P_*(T) u_*(T)$, where $e_*\ge 0$ is an integer, $P_*(T)\in\OO_K[T]$ is a Weierstraß polynomial, and $u_*(T)\in\Omega^\times$ is a unit.
    Without loss of generality, we may assume that $\min\{e_a,e_b,e_c,e_d\}=0$ and $\gcd(P_a,P_b,P_c,P_d)=1$. Furthermore, write $\det(A)=\pi^e P(T) u(T)$ with $e$, $P$, $u$ as above. Let $Q(T)\in\OO_K[T]$ be a Weierstraß polynomial coprime to $P(T)$, so that $\gcd(\det(A),Q(T))=1$. 

    We claim that $\det(A)\in\Omega^\times$.
    Let $\begin{psmallmatrix} x & y \\ z & w\end{psmallmatrix}\in\LLambda$. Then $\LLambda\subseteq A M_2(\Omega) A^{-1}$ implies that
    \[A^{-1} \begin{pmatrix} x & y \\ z & w\end{pmatrix} A = \frac{1}{\det A} \begin{pmatrix} adx+cdy-abz-bcw & bdx+d^2y-b^2z-bdw \\ -acx-c^2y+a^2z+acw & -bcx-cdy+abz+adw \end{pmatrix} \in M_2(\Omega).\]
    Set $y\colonequals Q(T)$ and set the rest of the variables to be zero, or $z\colonequals1$ and the rest of them zero.
    Since the matrix above must be in $M_2(\Omega)$, this gives us the following divisibilities:
    \begin{align*}
        y=Q(T),\, x=z=w=0 &\implies \det(A) \mid cdQ(T),d^2Q(T),c^2Q(T) \implies \det(A) \mid cd,d^2,c^2; \\
        z=1,\, x=y=w=0 &\implies \det(A) \mid ab,b^2,a^2.
    \end{align*}
    In the first line, the second implication is because $\gcd(\det(A),Q(T))=1$.
    In particular, we have
    \[\det(A)\mid a^2,b^2,c^2,d^2.\]
    If $e>0$, then this shows $\pi\mid a,b,c,d$, which contradicts the assumption $\min\{e_a,e_b,e_c,e_d\}=0$, so we get that $e=0$. Similarly, we must have $P(T)=1$, otherwise $\gcd(P_a,P_b,P_c,P_d)=1$ fails. So $\det(A)$ is a unit, as claimed.

    The maximal order $A M_2(\Omega) A^{-1}$ consists of the following matrices, with $x,y,z,w\in\Omega$:
    \[A \begin{pmatrix} x & y \\ z & w \end{pmatrix} A^{-1}=\frac{1}{\det A}\begin{pmatrix} adx -acy + bdz -bcw & a^2y -ab(x+w) - b^2z \\ -c^2y+ cd(x-w) + d^2z & -bcx + acy - bdz+adw\end{pmatrix}.\]
    Since $\det(A)\in\Omega^2$, this shows that the top right entries of matrices in $A M_2(\Omega) A^{-1}$ run through the ideal $(a^2,ab,b^2)=(a,b)^2$. Now if 
    \[\LLambda=\bigcap_{i=1}^r A_i M_2(\Omega) A_i^{-1}\]
    with $A_i=\begin{psmallmatrix} a_i & b_i \\ c_i & d_i\end{psmallmatrix}\in M_2(\Omega)\cap \GL_2(\mathscr D)$ as above, then the elements in the top right corner run through the ideal
    \[\bigcap_{i=1}^r (a_i,b_i)^2=\left(\bigcap_{i=1}^r (a_i,b_i)\right)^2.\]
    This is the square of an ideal in $\Omega$. But $\mm$ is not a square, so $\LLambda$ cannot be written as an intersection of maximal orders. \qed
\end{example}

\begin{proposition}
    {Suppose that $R$, $\mathscr D$, and $\Omega$ are as in \cref{sec:standard-form}, let $\mm_\Omega$ be as in \cref{rem:Omega-local}, and let $\mathbf n$ and $\mathbf I$ be as in \cref{def:standard-form}. Assume that $R$ is complete in the $\mm_R$-adic topology, where $\mm_R$ is the maximal ideal of $R$.}
    Then an $R$-order $\LLambda\subset \mathscr A=M_n(\mathscr D)$ is graduated if
    \begin{enumerate}
        \item $\LLambda/\Jac(\LLambda)\simeq \bigoplus_{i=1}^t M_{n_i}(\Omega/\mm_\Omega)$ with $\sum_{i=1}^t n_i=n$, and
        \item $\LLambda$ arises as a finite intersection of maximal $R$-orders in $\mathscr A$.
    \end{enumerate}
    More precisely, $\LLambda\subset\mathscr A$ is a graduated $R$-order if and only if
    \begin{enumerate}
        \item $\LLambda/\Jac(\LLambda)\simeq \bigoplus_{i=1}^t M_{n_i}(\Omega/\mm_\Omega)$ with $\sum_{i=1}^t n_i=n$, and
        \item $\LLambda$ contains a graduated order $\LLambda'\subseteq \LLambda$ that arises as an intersection of maximal $R$-orders in $\mathscr A$.
    \end{enumerate}
\end{proposition}
\begin{proof}
    The first statement follows directly from \cref{def:graduated}, exactly as in \cite[`(iv) $\Rightarrow$ (i)', p.~20]{Plesken}; {this is where the completeness hypothesis is used}.
    For the second statement, first suppose that $\LLambda$ is graduated. By \cref{every-standard-form}, we may assume $\LLambda$ to be of standard form; then it is clear from \cref{II.4.ii} that $\LLambda/\Jac(\LLambda)$ is as in the statement. Now let $\LLambda'=\LLambda(\mathbf n,\mathbb I')\subseteq\LLambda$ be as in \cref{lem:principalisation}. Then 
    \[\LLambda(\mathbf n,\mathbf I')=\bigcap_{1\le i,j\le n} \LLambda(\mathbf n,\mathbf A(i,j)),\]
    where for $1\le k,\ell\le n$,
    \[\mathbf A(i,j)_{k,\ell}\colonequals\begin{cases}
        I'_{ij} & \text{if $i=k$ and $j=\ell$}, \\
        \Omega & \text{otherwise}.
    \end{cases}\]
    As in \eqref{eq:intersection-principal}, the orders $\LLambda(\mathbf n,\mathbf A(i,j))$ can be written as intersections of maximal orders. It follows that $\LLambda'$ is an intersection of finitely many maximal orders.
    Conversely, if $\LLambda'$ is a graduated order contained in $\LLambda$, then $\LLambda$ is graduated by \cref{lem:contained-graduated-order}.
\end{proof}

\subsubsection{The $(\cap,+)$-lattice of $\LLambda$-lattices} \phantom-\medskip

\noindent
{Suppose that $R$ and $\mathscr D$ are as in \cref{sec:standard-form}.}
Let $\LLambda\subset \mathscr A=M_n(\mathscr D)$ be an $R$-order and $V$ an irreducible $\mathscr A$-module.
As before, let $\mathfrak Z(V)$ denote the set of nonzero left $\LLambda$-lattices in $V$.
The set $\mathfrak Z(V)$ is a $(\cap,+)$-lattice; we refer to \cite[\S2]{DaveyPriestley} for background on the theory of $(\cap,+)$-lattices.
In \cite[§II.c]{Plesken}, a Galois correspondence is established between graduated $R$-orders in $\mathscr A$ containing $\LLambda$ and so-called admissible distributive sublattices of $\mathfrak Z(V)$ in the {one}-dimensional case. One direction of this correspondence is given by 
$\GGamma \mapsto \{\text{nonzero $\GGamma$-lattices in $V$}\}$.

This approach, however, fails in dimension $2$. For this, recall that the $\mathbf M_3$--$\mathbf N_5$ theorem \cite[\nopp 4.10]{DaveyPriestley} shows that a $(\cap,+)$-lattice containing the pentagon $\mathbf N_5$ as a $(\cap,+)$-sublattice fails to be distributive (actually, it isn't even modular).

\begin{example}
    Consider the setup of \cref{sec:hereditary}. For simplicity, let $d=1$, so that $\mathscr D=K$ and $\Omega\colonequals\OO_k[[T]]$, where $k/\QQ_p$ is a finite extension, and let $n=1$, so that $\mathscr A=D=\Frac(\Omega)\simeq V$. Then $\LLambda\colonequals\GGamma\colonequals\Omega$ is {a maximal $R$-order} in $\mathscr A$, so in particular, it is graduated. The $\Omega$-lattices in $\mathscr A$ are then the nonzero fractional ideals of $\Omega$.
    This contains $\mathbf N_5$ as a $(\cap,+)$-sublattice:
    
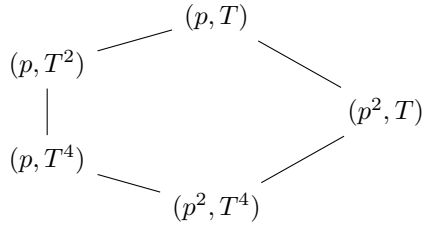
\[\begin{tikzcd}[row sep=0mm]
	& {(p,T)} \\
	{(p,T^2)} \\
	&& {(p^2,T)} \\
	{(p,T^4)} \\
	& {(p^2,T^4)}
	\arrow[no head, from=1-2, to=2-1]
	\arrow[no head, from=1-2, to=3-3]
	\arrow[no head, from=2-1, to=4-1]
	\arrow[no head, from=3-3, to=5-2]
	\arrow[no head, from=4-1, to=5-2]
    \end{tikzcd}\]
    This is an example of the general fact that a noetherian integral domain has distributive $(\cap,+)$-lattice of ideals iff it is a Dedekind domain, which $\OO_k[[T]]$ is not. Note that this fact further generalises to noncommutative domains \cite{Brungs}.
\end{example}

\subsubsection{Extremal orders} \phantom-\medskip

\noindent
{Let $R$ be a noetherian integral domain with field of fractions $L=\Frac(R)$, and let $\mathscr A$ be a separable $L$-algebra.}

\begin{definition}
    Let $\LLambda, \GGamma\subset \mathscr A$ be $R$-orders in the separable $L$-algebra $\mathscr A$. Then $\GGamma$ \emph{radically covers} $\LLambda$ if $\GGamma\supseteq \LLambda$ and $\Jac(\GGamma)\supseteq\Jac(\LLambda)$. We write $\GGamma\succeq\LLambda$ in this case. The order $\LLambda$ is called \emph{extremal} if for all $R$-orders $\GGamma$, $\GGamma\succeq\LLambda$ implies $\GGamma=\LLambda$.
\end{definition}
{When $R$ is a complete discrete valuation ring,} extremal orders are precisely the hereditary ones, see \cite[\S39]{MO}. {In fact, the completeness hypothesis can be relaxed, see \cite[364]{MO}.} 

Every extremal order is the direct sum of summands that are extremal orders in the simple components of $\mathscr A$, and conversely, every extremal order arises in this way: this is \cite[Theorem~39.13]{MO}, and the proof given there works whenever $R$ is a commutative local ring. In op.cit., the theory of extremal orders is built up from 39.2--39.10 in full generality, and hence applies in our setting without further modification. We readily obtain the following:

\begin{proposition} \label{extremal-orders}
    {Suppose that $R$, $\mathscr D$, and $\Omega$ are as in \cref{sec:standard-form}, and let $\mm_\Omega$ be as in \cref{rem:Omega-local}.}
    Then every extremal $R$-order $\LLambda\subset\mathscr A=M_n(\mathscr D)$ is isomorphic to some $\LLambda(\mathbf n, \mathbf I)$ with
    \begin{equation} \label{eq:shape-of-I}
        \mathbf I=\begin{psmallmatrix}
        \Omega & \mm_\Omega & \mm_\Omega & \dots & \mm_\Omega \\
        \Omega & \Omega & \mm_\Omega & \dots & \mm_\Omega \\
        \vdots & \vdots & \ddots & \ddots & \vdots \\
        \Omega & \Omega & \dots & \Omega & \mm_\Omega \\
        \Omega & \Omega & \dots & \Omega & \Omega \\
    \end{psmallmatrix}.
    \end{equation}
    In particular, extremal orders are graduated. Conversely, any $R$-order $\LLambda(\mathbf n, \mathbf I)$ with $\mathbf I$ as above is extremal.
\end{proposition}
\begin{proof}
    We only need to modify the proof of \cite[Theorem~39.14]{MO} at a few places.
    
    First let $\LLambda\subset \mathscr A$ be an extremal $R$-order, and let $\LLambda\subseteq\GGamma\subset\mathscr A$ be a maximal overorder. It follows that $\Jac(\GGamma)\subseteq \Jac(\LLambda)$, so $\overline\LLambda\colonequals\LLambda/\Jac(\LLambda)\le \overline\GGamma\colonequals \GGamma/\Jac(\GGamma)$ is a subring. In Reiner's book, this is proven through a series of exercises (he provides sketches of proofs), and the proofs work in our setting as well.
    
    Every maximal order in $\mathscr A$ is the $\Omega$-endomorphism ring of a full free right $\Omega$-lattice $L$: in the {one}-dimensional case, this is \cite[Corollary~17.4]{MO} {for complete {discrete valuation} rings, and the completeness hypothesis can be removed by \cite[Remark~18.8(ii)]{MO}.} In {the two-dimensional} setting, the same proof works by using \cref{lem:all-conjugates-maximal} for the description of maximal orders. So $\GGamma\simeq \End_\Omega(L)$, which allows one to identify $\GGamma/\Jac(\GGamma)\simeq \End_{\Omega/\mm_\Omega}(L/\mm_\Omega L)\simeq M_n(\Omega/\mm_\Omega)$, and show that $\overline\LLambda\le M_n(\Omega/\mm_\Omega)$ is an extremal subring by Reiner's argument. Extremal subrings of $M_n(\Omega/\mm_\Omega)$ are described explicitly in \cite[\nopp 39.6 \& 39.10]{MO}, and lifting back to $\GGamma$ shows the first assertion. The second assertion follows immediately.

    Conversely, suppose that $\LLambda\colonequals\LLambda(\mathbf n, \mathbf I)\preceq \GGamma$ for some $R$-order $\GGamma\subset\mathscr A$.
    For all $1\le i\le n$, $\mathsf f_i \GGamma\mathsf f_i\subset \mathsf f_i \mathscr A\mathsf f_i= \mathscr D$ is an $R$-order. 
    Since $\LLambda\subseteq\GGamma$ and {since $\Omega$ is a maximal $R$-order in $\mathscr D$}, it follows that $\mathsf f_i\LLambda\mathsf f_i\subseteq \mathsf f_i \GGamma\mathsf f_i=\Omega$. So $\GGamma$ is a graduated order with respect to the indecomposable idempotents $\mathsf f_i$, that is, $\GGamma$ is a graduated $R$-order in standard form (\cref{rem:graduated-wrt-f}).
    Now it follows from $\Jac(\LLambda)\subseteq\Jac(\GGamma)$ and \cref{II.4.i} that $\LLambda=\GGamma$. 
\end{proof}

\begin{remark}
    {It follows that hereditary orders are graduated. By \cite[Theorem~10.8]{MO}, it suffices to treat the case when $\mathscr A$ is simple. In this case, hereditary orders are extremal by \cite[Corollary~39.12]{MO}, and extremal orders are graduated by \cref{extremal-orders}.}
\end{remark}

\begin{corollary} \label{no-hereditary-orders}
    {Retain the setup of \cref{extremal-orders}. If $R$ has non-principal maximal ideal $\mm_R$, and $R$ is complete in the $\mm_R$-adic topology, then there are no hereditary $R$-orders in $\mathscr A=M_n(\mathscr D)$.}
\end{corollary}
\begin{proof}
    Let $\LLambda\subset \mathscr A$ be an $R$-order. Then $\LLambda$ is hereditary if any only if $\Jac(\LLambda)$ is an invertible $(\LLambda,\LLambda)$-bimodule: this is shown in Step~2 in the proof of \cite[Theorem~39.1]{MO}; {this is where the assumption on completeness is used}. Furthermore, it follows from loc.cit. that hereditary orders are always extremal, see \cite[Corollary~39.12]{MO}. So it suffices to show that $\LLambda(\mathbf n, \mathbf I)$ has non-invertible Jacobson radical, where $\mathbf I$ is as in \cref{extremal-orders}. By \cref{II.4.i}, we have $\Jac(\LLambda(\mathbf n, \mathbf I))=\LLambda(\mathbf n, \mathbf I*\mm_\Omega\mathbf 1)$. 
    
    Recall the following general fact: in a ring $\mathfrak R$ with finitely many two-sided maximal ideals $\mm_1,\ldots,\mm_k$, a two-sided ideal $\mathfrak I$ is invertible if and only if it is principal and generated by a non-zero-divisor. For a proof, see \cite[Theorem~60]{Kaplansky}; note that commutativity of $\mathfrak R$ is not necessary.
    
    The matrix ring $M_n(\Omega)$ has only finitely many maximal ideals, so the previous fact is applicable. It remains to show that the two-sided ideal $\LLambda(\mathbf n, \mathbf I*\mm_\Omega\mathbf 1)$ isn't principal. Suppose it is, so $\LLambda(\mathbf n, \mathbf I*\mm_\Omega\mathbf 1)=X \LLambda(\mathbf n, \mathbf I)$ for some $X\in \LLambda(\mathbf n, \mathbf I)$. Consider the ideal generated by the reduced norms of its elements:
    \begin{equation} \label{eq:norms-L}
        \nr_{\mathscr A/L}\left(\LLambda(\mathbf n, \mathbf I*\mm_\Omega\mathbf 1)\right) = \nr_{\mathscr A/L}\left(X \LLambda(\mathbf n, \mathbf I)\right) = \nr_{\mathscr A/L}(X) \cdot \nr_{\mathscr A/L}\left(\LLambda(\mathbf n, \mathbf I)\right).
    \end{equation}
    We have $\nr_{\mathscr A/L}\left(\LLambda(\mathbf n, \mathbf I)\right)=R$: indeed, $\LLambda(\mathbf n, \mathbf I)$ is an order, hence the reduced norms of its elements are in $R$ by \cite[Theorem~10.1]{MO}; and on the other hand, since it's a graduated order, it contains all diagonal matrices $\diag(r,1,\ldots,1)$ with $r\in R$, and we have $\nr_{\mathscr A/L}(\diag(r,1,\ldots,1))=r$. So the right hand side of \eqref{eq:norms-L} is a principal ideal of $R$.

    Since $\mathbf I$ is of the form \eqref{eq:shape-of-I}, every matrix in $\LLambda(\mathbf n, \mathbf I*\mm_\Omega\mathbf 1)$ has reduced norm in $\mm_R$. On the other hand, taking the reduced norm of diagonal matrices in $\LLambda(\mathbf n, \mathbf I*\mm_\Omega\mathbf 1)\cap M_n(R)$ with non-zero diagonal entries, we see that $\nr_{\mathscr A/L}\left(\LLambda(\mathbf n, \mathbf I*\mm_\Omega\mathbf 1)\right)$ contains the non-principal ideal $\mm_R^n$, and hence cannot be principal.
\end{proof}

\subsubsection{Graduated hulls} \phantom-\medskip

\noindent
{Let $R$ be a noetherian integral domain with field of fractions $L=\Frac(R)$, and let $\mathscr A$ be a separable $L$-algebra.}

\begin{definition}
    Let $\LLambda\subset \mathscr A$ be an $R$-order in a separable $L$-algebra $\mathscr A$. A \emph{graduated hull} of $\LLambda$ is a graduated $R$-order $\GGamma\subset \mathscr A$ such that $\LLambda\preceq\GGamma$ and such that if $\GGamma'\subset\mathscr A$ is another $R$-order with $\LLambda\preceq\GGamma'$, then $\GGamma\subseteq\GGamma'$.
\end{definition}
\begin{lemma} \label{lem:intersections}
    {Suppose that $R$, $\mathscr D$, and $\Omega$ are as in \cref{sec:standard-form}, let $\mm_\Omega$ be as in \cref{rem:Omega-local}, and let $\mathbf n_1,\mathbf n_2$ and $\mathbf I_1,\mathbf I_2$ be as in \cref{def:standard-form}.}
    Graduated orders in standard form admit the following properties:
    \begin{enumerate}
        \item $\LLambda(\mathbf n,\mathbf I_1) \cap \LLambda(\mathbf n,\mathbf I_2)=\LLambda(\mathbf n, \mathbf I_1\cap\mathbf I_2)$;
        \item $\Jac(\LLambda(\mathbf n,\mathbf I_1)) \cap \Jac(\LLambda(\mathbf n,\mathbf I_2))=\Jac(\LLambda(\mathbf n, \mathbf I_1\cap\mathbf I_2))$;
        \item $\LLambda(\mathbf n_1,\mathbf I_1) \cap \LLambda(\mathbf n_2,\mathbf I_2)=\LLambda(\mathbf n_1\cap \mathbf n_2,\mathbf I_1\cap \mathbf I_2)$;
        \item $\Jac(\LLambda(\mathbf n_1,\mathbf I_1)) \cap \Jac(\LLambda(\mathbf n_2,\mathbf I_2))=\Jac(\LLambda(\mathbf n_1\cap\mathbf n_2, \mathbf I_1\cap\mathbf I_2))$.
    \end{enumerate}
    In the first two equalities, the ideal matrix $\mathbf I_1\cap \mathbf I_2$ is obtained by taking the entry-wise intersections of $\mathbf I_1$ and $\mathbf I_2$.
    Considering $\mathbf n_1$ and $\mathbf n_2$ as subdivisions of a line segment of length $n$ into segments of integer length, we let $\mathbf n_1\cap \mathbf n_2$ denote the join of both subdivisions. The ideal matrix $\mathbf I_1\cap \mathbf I_2$ in the latter statements is obtained by first subdividing the ideal matrices $\mathbf I_1$ and $\mathbf I_2$ along $\mathbf n_1\cap \mathbf n_2$, and then taking the entry-wise intersection.
\end{lemma}
\begin{proof}
    All of these statements follow from \cref{def:standard-form} and \cref{II.4.i}.
\end{proof}
\begin{corollary}
    {Suppose that $R$, $\mathscr D$, and $\Omega$ are as in \cref{sec:standard-form}. Then every $R$-order $\LLambda\subset \mathscr A=M_n(\mathscr D)$ has a graduated hull.}
\end{corollary}
\begin{proof}
    We show that there is an extremal $R$-order $\widetilde\GGamma$ such that $\LLambda\preceq \widetilde\GGamma$. Let $\LLambda\equalscolon\GGamma_0\preceq\GGamma_1\preceq\ldots$ be a chain of orders radically covering each other. Let $\GGamma_\infty\colonequals\bigcup_{i=0}^\infty \GGamma_i$ be their directed union. Every element of $\GGamma_\infty$ is contained in some $\GGamma_i$ and hence integral over $R$ by \cite[Theorem~10.3]{MO}, and $\LLambda\subseteq \GGamma_\infty$ implies $L\GGamma_\infty=\mathscr A$, so $\GGamma_\infty$ is also an $R$-order by loc.cit. In particular, $\GGamma_\infty$ is noetherian, so the ascending chain of submodules $\GGamma_i$ stablilises. The existence of extremal orders follows.
    
    By \cref{extremal-orders}, $\widetilde\GGamma$ is graduated, and without loss of generality, we may assume it to be in standard form {(\cref{every-standard-form})}; note that then any graduated order $\LLambda\subseteq \GGamma\subseteq \widetilde\GGamma$ is also in standard form. 
    It follows from \cref{II.4.i} that if $\LLambda\preceq \GGamma,\GGamma'$ for some graduated orders $\GGamma,\GGamma'$ in standard form, then $\GGamma\cap\GGamma'$ is also a graduated order in standard form, and $\LLambda\preceq\GGamma\cap\GGamma'$ by \cref{lem:intersections}. The existence of graduated hulls follows.    
\end{proof}

\subsection{Self-dual orders}
In Plesken's generalisation of Jacobinski's conductor formula, the group ring is replaced by a self-dual order {(the definition will be given momentarily)}. We adapt this approach to completed group algebras.
{Let $R$ be a noetherian integrally closed integral domain, let $L=\Frac(R)$ be its field of fractions, and let $\mathscr A$ be a separable $L$-algebra.}

\begin{definition} \label{def:self-dual}
    Let $M\subset \mathscr A$ be a full left $\LLambda$-lattice. 
    Then $M$ is called \emph{reflexive} if the canonical map $M\to\Hom_R(\Hom_R(M,R),R)$ is an isomorphism, and \emph{self-dual} if $\D(M/R)=M$.
\end{definition}

\begin{remark}
    Plesken \cite[Definition~III.1]{Plesken} considers the following slightly diffent notion: for a unit $u\in\cent(\mathscr A)^\times$, let $\D_u(M/R)\colonequals \{a\in\mathscr A: \tr_{\mathscr A/L}(u a M) \subseteq R \}$. Then $\LLambda\subset \mathscr A$ is called \emph{$u$-self-dual} if there exists some $u\in\cent(\mathscr A)^\times$ such that $\D_u(\LLambda/R)=\LLambda$. 
    For $u=1$, one recovers $\D_1=\D$.
    Note that the same proof as in \cite[Proposition~III.6]{Plesken} shows the following: if $\LLambda\subset \mathscr A$ is a self-dual $R$-order and $\epsilon\in\LLambda$ is an idempotent, then $\epsilon\LLambda\epsilon\subset \epsilon \mathscr A\epsilon$ is also a self-dual order.
\end{remark}

\begin{lemma} \label{lem:double-duals}
    Let $\LLambda\subset \mathscr A$ be an $R$-order. Then the ordinary and reduced double duals agree:
    \[\D_{\ord}(\D_{\ord}(\LLambda/R)/R) = \D(\D(\LLambda/R)/R).\]
\end{lemma}
\begin{proof}
    We may assume that $\mathscr A$ is simple, so $\mathscr A=M_n(\mathscr D)$ is of index $s_\mathscr A$. Then $\Tr_{\mathscr A/K}=s_{\mathscr A} \tr_{\mathscr A/K}$ by the properties of the reduced trace \cite[(7.35)]{CR}.
    It follows from the definitions that $\D_{\ord}(\LLambda/R)=\frac{1}{s_{\mathscr A}}\D(\LLambda/R)$. 
    By linearity of the ordinary trace map, $\D_{\ord}(aM/R)=a^{-1}\D_{\ord}(M/R)$ for any full left $R$-lattice $M\subset \mathscr A$ and $a\in K^\times$.
    So the ordinary double dual of $\LLambda$ is
    \[\D_{\ord}(\D_{\ord}(\LLambda/R)/R)=\D_{\ord}\left(\frac{1}{s_{\mathscr A}}\D(\LLambda/R) \middle/ R\right)=s_{\mathscr A}\D_{\ord}(\D(\LLambda/R)/R)=s_{\mathscr A} \cdot \frac{1}{s_{\mathscr A}}\D(\D(\LLambda/R)/R),\]
    as claimed.
\end{proof}

\begin{lemma} \label{lem:selfdual-reflexive}
    Let $\LLambda\subset \mathscr A$ be a self-dual $R$-order. Then $\LLambda$ is reflexive.
\end{lemma}
\begin{proof}
    Using \cref{lem:double-duals}, the ordinary double dual of $\Lambda$ is
    \[\D_{\ord}(\D_{\ord}(\LLambda/R)/R)=\D(\D(\LLambda/R)/R)= \D(\LLambda/R)=\LLambda.\]
    Then \cite[Proposition~2.2]{NickelConductor} shows that $\LLambda$ is reflexive.
\end{proof}

\begin{lemma} \label{lem:conductor-reflexive}
    {Suppose that $R$ is a regular local ring of dimension at most two.}
    Let $\LLambda\subset \mathscr A$ be a self-dual $R$-order, and let $\LLambda\subseteq \GGamma\subset \mathscr A$ be an $R$-overorder. Then the conductors $(\GGamma:\LLambda)_r$ and $(\GGamma:\LLambda)_\ell$ are reflexive $R$-modules.
\end{lemma}
\begin{proof}
    The proof is similar to \cite[Proposition~2.7]{NickelConductor}. Consider the natural map $(\GGamma:\LLambda)_r\to \bigcap_\pp (\GGamma:\LLambda)_{r,\pp}$, where $\pp$ ranges over the prime ideals of $R$ of height $1$. This {natural map is injective: this is obvious when $R$ is a discrete valuation ring, and follows from \cite[\nopp 5.1.2 \& 5.1.8]{NSW} in the two-dimensional case}. Suppose that $a\in (\GGamma:\LLambda)_{r,\pp}$ for all such $\pp$. Then we have
    \[\GGamma a \subseteq \bigcap_\pp \GGamma_\pp a = \bigcap_\pp \LLambda_\pp = \LLambda,\]
    where the last equality is \cref{lem:selfdual-reflexive}. So $a\in(\GGamma:\LLambda)_r$, and thus the right conductor is reflexive. The proof is analogous in the left case.
\end{proof}

The following is an adaptation of \cite[Proposition~III.7]{Plesken} and a generalisation of \cite[Theorem~2.8]{NickelConductor}.
\begin{proposition} \label{selfdual-conductors}
    {Suppose that $R$ is a regular local ring of dimension at most two.}
    Let $\LLambda\subset \mathscr A$ be a self-dual $R$-order, and let $\LLambda\subseteq \GGamma\subset \mathscr A$ be an $R$-overorder of $\LLambda$. Then
    \[(\GGamma:\LLambda)_r = (\GGamma:\LLambda)_\ell = \D(\GGamma/R);\]
    in particular, the conductor is independent of $\LLambda$.
\end{proposition}
\begin{proof}
    The right conductor $(\GGamma:\LLambda)_r$ is the largest left $\GGamma$-lattice in $\LLambda$, hence its $\D$-dual is the smallest right $\GGamma$-lattice containing $\D(\LLambda/R)$. In other words,
    \[\D((\GGamma:\LLambda)_r / R) = \D(\LLambda/R)\GGamma = \LLambda \GGamma = \GGamma.\]
    Therefore the double dual of $(\GGamma:\LLambda)_r$ is
    $\D(\D((\GGamma:\LLambda)_r/R)/R)=\D(\GGamma/R).$ By \cref{lem:double-duals}, this agrees with the ordinary double dual, and by \cref{lem:conductor-reflexive} and \cite[Proposition~2.2]{NickelConductor}, the latter is $(\GGamma:\LLambda)_r$ itself. The proof is analogous for the left conductor.
\end{proof}

\begin{lemma} \label{different-direct-sum}
    {Let $R$ be a noetherian integrally closed integral domain, let $L=\Frac(R)$ be its field of fractions, and let $\mathscr A$ be a separable $L$-algebra.}
    Let $\LLambda\subset \mathscr A$ be a self-dual $R$-order, and let $\LLambda\subseteq \GGamma\subset \mathscr A$ be an $R$-overorder of $\LLambda$. 
    Let $\GGamma_\chi\colonequals \GGamma\cap\mathscr A_\chi$, and suppose that $\GGamma=\bigoplus_{\chi\in\mathscr X}\GGamma_\chi$.
    Then the conductor of $\GGamma$ into $\LLambda$ is
    \[\D(\GGamma/R) = \bigoplus_{\chi\in \mathscr X} \D_{\mathscr A_\chi / L}(\GGamma_\chi / R).\]
\end{lemma}
\begin{proof}
    The proof is the same as the first half of \cite[Theorem~III.8]{Plesken}. 
    Let $a=\sum_{\chi\in \mathscr X}a_\chi\in\mathscr A$: then the following equivalences hold.
    \begin{align*}
        a\in\D(\GGamma/R) &\iff \forall y=\sum_{\chi\in X}y_\chi\in \GGamma: \sum_{\chi\in \mathscr X}\tr_{\mathscr A_\chi / L}(a_\chi y_\chi) \in R \\
        &\iff \forall\chi\in \mathscr X, \forall y_\chi\in \GGamma_\chi: \tr_{\mathscr A_\chi / L}(a_\chi y_\chi)\in R \\
        &\iff \forall\chi\in \mathscr X, a_\chi \in \D_{\mathscr A_\chi / L}(\GGamma_\chi/ R) \qedhere
    \end{align*}
\end{proof}

\begin{corollary} \label{general-conductor-formula}
    {Suppose that $R$ is a regular local ring of dimension at most two, let $\mathscr A$ be a separable $L=\Frac(R)$-algebra.}
    Let $\LLambda\subset \mathscr A$ be a self-dual $R$-order, and let $\LLambda\subseteq \GGamma\subset \mathscr A$ be a graduated $R$-overorder of $\LLambda$.
    Suppose that $\GGamma=\bigoplus_{\chi\in\mathscr X}\GGamma_\chi$.
    In the $\chi$-part, let $\mathbf E_{\chi}$ denote the $n_\chi\times n_\chi$ matrix consisting entirely of ones, and suppose $\GGamma_\chi\simeq \LLambda(\mathbf n_\chi, \mathbf I_\chi)$.
    Then the conductor of $\GGamma$ into $\LLambda$ is
    \[\D(\GGamma/R) \simeq \bigoplus_{\chi \in \mathscr X} \LLambda\left(\mathbf n_\chi, \D(\Omega_\chi/R)\mathbf E_\chi * \mathbf I_\chi^{-,\top}\right).\]
\end{corollary}
\begin{proof}
    {This is a direct} consequence of \cref{different-graduated,selfdual-conductors,different-direct-sum}.
\end{proof}

\section{Maximal orders in completed group rings} \label{sec:maximal}
\subsection{Wedderburn decomposition of \texorpdfstring{$\Q^F(\G)$}{Q\^F(G)}} \label{sec:Wedderburn} 

Let $p$ be an odd rational prime. Let $H$ be a finite group and $\Gamma\simeq \ZZ_p$ be isomorphic to the additive group of the $p$-adic integers. Let $\G=H\rtimes \Gamma$ be a semidirect product: in particular, $\G$ is a {one}-dimensional $p$-adic Lie group. Let $F/\QQ_p$ be a finite extension, and write $\Lambda^{\OO_F}(\G)\colonequals\OO_F\llb\G\rrb$ for the completed group algebra. Let $\Q^F(\G)\colonequals \Quot(\Lambda^{\OO_F}(\G))$ be its total ring of {quotients}. By work of Ritter--Weiss \cite[Proposition~5(1)]{TEIT-II}, this is a semisimple artinian ring. Fix an integer $n_0$ that is large enough such that $\Gamma_0\colonequals\Gamma^{p^{n_0}}$ is central in $\G$.

The ring $\Q^F(\G)$ has been studied by a number of authors: Ritter--Weiss \cite[Proposition~6(2)]{TEIT-II} provided a description of its centre when $F$ is large enough, and their result was generalised by Nickel \cite[\S1]{NickelConductor} to arbitrary $F$. In the case of $\G$ pro-$p$, Lau \cite{Lau} described the skew fields explicitly, and these results have been generalised to arbitrary $\G$ by the author.
In the rest of this subsection, we summarise the main results of \cite{W}.

Consider the Wedderburn decomposition of the completed group ring $\Q^F(\G)$: the Artin--Wedderburn theorem provides an {implicit} isomorphism 
\begin{equation} \label{eq:Wedderburn}
    \Q^F(\G)\simeq \bigoplus_{\chi\in \Irr(\G)/\sim_F} M_{n_\chi} (D_\chi){.}
\end{equation}
Here $\Irr(\G)$ is the set of characters of $\G$ with open kernel, and $\chi\sim_F\chi'$ if there exists some $\sigma\in\Gal(F_\chi/F)$ such that $\sigma(\res^\G_H\chi)=\chi'$, where $F_\chi\colonequals F(\chi(h):h\in H)$.
{In \cref{thm:Wedderburn} below, the skew fields $D_\chi$ occurring on the right hand side will be described explicitly in terms of the Wedderburn decomposition of the group algebra $F[H]$. Before stating this result, we need to make some preparations.}

{The group algebra $F[H]$ admits a decomposition of the form}
\[F[H]\simeq \bigoplus_{\eta\in\Irr(H)/\sim_F} M_{n_\eta}(D_\eta).\]
On the right hand side, the skew field $D_\eta$ has centre $F(\eta)\colonequals F(\eta(h):h\in H)$, and we write $\dim_{F(\eta)} D_\eta=s_\eta^2$ for the square of the Schur index. The skew fields $D_\eta$ are well understood: they are determined by their Hasse invariants \cite[\S14+31]{MO}. {More explicitly, $D_\eta$ can be identified with the cyclic algebra
\begin{equation} \label{eq:D-eta-cyclic}
    D_\eta=\left(W/F(\eta),\sigma,\pi\right) = \bigoplus_{i=0}^{s_\eta-1} W \pi^i,
\end{equation}
where $W$ is the unique unramified extension of $F(\eta)$ of degree $s_\eta$, $\sigma$ is a certain power of the Frobenius, $\pi$ is an $s_\eta$th root of a fixed uniformiser $\pi_\eta$ of $F(\eta)$, and conjugation by $\pi$ acts as $\sigma$ on $W$.
}

If $\G=H\times\Gamma$ is a direct product, then $\Lambda^{\OO_F}(\G)=\Lambda^{\OO_F}(\Gamma)[H]$, so we simply have $D_\chi=\Q^F(\Gamma)\otimes_F D_\eta$, where $\eta=\res^\G_H\chi$. In the case of an arbitrary semidirect product $\G=H\rtimes\Gamma$, the situation is significantly more complicated.

For $\chi\in \Irr(\G)$ and $\eta\in\Irr(H)$, there are the following classical idempotents over $H$:
\begin{align*}
	e(\eta) &\colonequals \frac{\eta(1)}{\#H} \sum_{h\in H} \eta(h^{-1}) h \in F(\eta)[H], &
	e_\chi &\colonequals \sum_{g\in\G/\G_\eta} e({}^g\eta) \in F_\chi[H].
\end{align*}
Here $\G_\eta$ denotes the stabliliser of $\eta$ in $\G$. For $\eta\mid\res^\G_H\chi$ an irreducible constituent, the index {$w_\chi\colonequals[\G:\G_\eta]$} depends only on $\chi$, {and it follows from the definitions that $F(\eta)\supseteq F_\chi$}. Summing over Galois orbits, these idempotents have the following counterparts defined over $F$:
\begin{align} \label{eq:epsilon-idempotents}
	\epsilon(\eta) &\colonequals \sum_{\sigma\in\Gal(F(\eta)/F)} e({}^\sigma\eta) \in F[H], &
	\epsilon_\chi &\colonequals \sum_{\sigma\in\Gal(F_{\chi}/F)} \sigma(e_\chi) \in F[H].
\end{align}
The primitive central idempotents of $\Q^F(\G)$ exactly those of the form $\epsilon_\chi$.

Let $\chi\in \Irr(\G)$, and let $\eta\mid\res^\G_H\chi$ be an irreducible constituent.
{In \cite[\S4.2]{W}, it is shown that $\Gal(F(\eta)/F_\chi)$ is cyclic, and a certain generator of it is described as follows. Let $v_\chi$ be the lowest integer such that conjugation by $\gamma^{v_\chi}$ acts as some automorphism $\tau\in\Gal(F(\eta)/F_\chi)$ on the character $\eta$; in formul\ae,
\[\tau(\eta(h))=\eta(\gamma^{v_\chi} h \gamma^{-v_\chi}) \quad \forall h\in H.\]
It is easily seen that $v_\chi\mid w_\chi$. It can be shown that there is exactly one such automorphism $\tau$, and that $\tau$ has order $w_\chi/v_\chi$. Furthermore, it is proven in \cite[\S2]{W} that $\tau$ admits a unique extension as an automorphism to $D_\eta$ with the same order, which we also denote by $\tau$.}

\begin{theorem}[{\cite[\S7]{W}}] \label{thm:Wedderburn}
The skew fields $D_\chi$ are described as follows.
\begin{theoremlist}
    \item The skew field $D_\chi$ has centre $\cent(D_\chi)=\Q^{F_\chi}(\Gamma''_\chi)$, where $\Gamma''_\chi\simeq \ZZ_p$ is a cyclic pro-$p$ group with topological generator $\gamma''_\chi$ (whose precise description we omit).
    \item The field $\Q^W(\Gamma''_\chi)$ is a maximal subfield of $D_\chi$.
    \item The skew field $D_\chi$ is explicitly given as the $\Q^W(\Gamma''_\chi)$-algebra
    \begin{align*}
        D_\chi = \bigoplus_{i=0}^{s_\eta {\frac{w_\chi}{v_\chi}}-1} \Q^W(\Gamma''_\chi) \cdot \left( \pi \cdot (\gamma''_\chi)^{{v_\chi/w_\chi}}\right)^i,
    \end{align*}
    where conjugation by $\pi \cdot (\gamma''_\chi)^{{v_\chi/w_\chi}}$ acts as $\sigma \tau$.
    \item {The maximal $\Lambda^{\OO_{F_\chi}}(\Gamma''_\chi)$-orders in $D_\chi$ are precisely the $D_\chi$-conjugates of}
    \[\Omega_\chi = \bigoplus_{i=0}^{s_\eta {\frac{w_\chi}{v_\chi}}-1} \Lambda^{\OO_W}(\Gamma''_\chi) \cdot \left( \pi \cdot (\gamma''_\chi)^{{v_\chi/w_\chi}}\right)^i,\]
    where $\OO_{D_\eta}$ is the unique maximal $\ZZ_p$-order in $D_\eta$.
\end{theoremlist}
\end{theorem}

\subsection{Comparing descriptions of the centre} \label{sec:conductor-formula}
Ritter--Weiss \cite[553-555]{TEIT-II} studied the $\chi$-part of the semisimple algebra $\Q^F(\G)$ by introducing an element $\gamma_\chi=c_\chi \gamma^{w_\chi}$ where $c_\chi\in (\QQ_p^\al[H]e_\chi)^\times$ such that $\gamma_\chi$ acts trivially on the vector space affording $\chi$. More precisely, they established an equality $\Q^F(\Gamma_\chi)= \cent(\Q^F(\G) e_\chi)$, where $\Gamma_\chi$ is the procyclic group topologically generated by $\gamma_\chi$, and where $F/\QQ_p$ is large enough such that $\chi$ has a realisation over it. 

Nickel \cite[Lemma~1.2]{NickelConductor} introduced a refinement $\gamma_\chi'\in\Q(\G) e_\chi$ of the Ritter--Weiss element $\gamma_\chi$. {Namely, Nickel constructed an element $\gamma'_\chi$, and showed that $\cent(\Q^F(\G) \epsilon_\chi)= \Q^{F_\chi}(\Gamma'_\chi)$, where $\Gamma_\chi'$ is the procyclic group topologically generated by $\gamma_\chi'$.
More explicitly, $\gamma'_\chi=x\gamma_\chi$, where $x\in U_E^1$ is a $1$-unit in some large enough finite extension $E/\QQ_p$. The element $\gamma'_\chi$ is characterised by the properties that $\gamma_\chi'$ is invariant under the $\Gal(E/F_\chi)$-action on $\cent(\Q^E(\G)e_\chi)$, and if $\gamma^{p^n}$ acts trivially on the representation affording $\chi$ for some $n\ge0$, then $x^{p^n}\in U_{F_\chi}^1$ is a $1$-unit in $F_\chi$.}

Nickel used this $\gamma'_\chi$ to study central conductors. In the next subsection, we will want to work with $\gamma''_\chi$ instead, which allows for more explicit methods through the use of \cref{thm:Wedderburn}. This essentially means working on the other side of the Wedderburn isomorphism \eqref{eq:Wedderburn}. We will show that as far as the central conductor is concerned, this makes no difference, by making compatible choices of embeddings on both sides. We remark that while the profinite groups $\Gamma'_\chi$ and $\Gamma''_\chi$ are abstractly isomorphic, a direct comparison of the elements $\gamma'_\chi$ and $\gamma''_\chi$ is rather difficult, due to a considerable difference {between} their definitions.

Let $\M$ be a maximal $\Lambda^{\OF}(\Gamma_0)$-order containing $\Lambda^{\OF}(\G)$. Fix an isomorphism $\phi_\chi:\Q^F(\G)\epsilon_\chi\to M_{n_\chi}(D_\chi)$. Then $\phi_\chi$ {maps} maximal orders to maximal orders. Since all maximal orders in $M_{n_\chi}(D_\chi)$ are conjugate to $M_{n_\chi}(\Omega_\chi)$ by \cref{lem:all-conjugates-maximal,thm:Wedderburn}, we may assume, by possibly post-composing $\phi_\chi$ with an inner automorphism of its target, that $\phi_\chi(\M\epsilon_\chi)=M_{n_\chi}(\Omega)$. On the centres, this induces $\phi_\chi(\cent(\M\epsilon_\chi))=\cent(M_{n_\chi}(\Omega))$. Then $\cent(\M\epsilon_\chi)\simeq \Lambda^{F_\chi}(\Gamma'_\chi)$ by \cite[Corollary~1.7]{NickelConductor}, and $\cent(M_{n_\chi}(\Omega))\simeq \Lambda^{F_\chi}(\Gamma''_\chi)$. So $\phi_\chi$ induces an isomorphism $\Lambda^{\OO_{F_\chi}}(\Gamma'_\chi) \simeq \Lambda^{\OO_{F_\chi}}(\Gamma''_\chi)$.

Recall that $\gamma_0=\gamma^{p^{n_0}}$ is a power of $\gamma$ such that the cyclic pro-$p$ group $\Gamma_0$ topologically generated by $\gamma_0$ is central in $\G$. Note that $\gamma_0\epsilon_\chi\in \Lambda^{\OO_{F_\chi}}(\Gamma'_\chi)$: indeed, we have $\gamma_\chi^{p^{n_0}/w_\chi}=\gamma_0 e_\chi$, and $x^{p^{n_0}}$ is a $1$-unit for $n_0\gg0$. This makes $\Lambda^{\OO_{F_\chi}}(\Gamma'_\chi)$ a ${\Lambda^{\OO_{F_\chi}}(\Gamma_0)}\epsilon_\chi$-algebra, and we also give $\Lambda^{\OO_{F_\chi}}(\Gamma''_\chi)$ a ${\Lambda^{\OO_{F_\chi}}(\Gamma_0)}\epsilon_\chi$-algebra structure through the isomorphism above.

{By the above, we have the following commutative diagram of embeddings and isomorphisms.}
\[\begin{tikzcd}[ampersand replacement=\&]
	{\Lambda^{\OO_{F_\chi}}(\Gamma_0)\epsilon_\chi} \& {\Lambda^{\OO_{F_\chi}}(\Gamma'_\chi)} \& {\cent(\mathfrak M\epsilon_\chi)} \& {\mathfrak M \epsilon_\chi} \& {\Q^F(\G)\epsilon_\chi} \\
	\& {\Lambda^{\OO_{F_\chi}}(\Gamma''_\chi)} \& {\cent\left(M_{n_\chi}(\Omega_\chi)\right)} \& {M_{n_\chi}(\Omega_\chi)} \& {M_{n_\chi}(D_\chi)}
	\arrow[hook, from=1-1, to=1-2]
	\arrow["\simeq", from=1-2, to=1-3]
	\arrow["\simeq", from=1-2, to=2-2]
	\arrow["\subseteq", hook, from=1-3, to=1-4]
	\arrow["\simeq", from=1-3, to=2-3]
	\arrow["\subset", from=1-4, to=1-5]
	\arrow["\simeq", from=1-4, to=2-4]
	\arrow["{\phi_\chi}"', "\simeq", from=1-5, to=2-5]
	\arrow["\simeq", from=2-2, to=2-3]
	\arrow["\subseteq", hook, from=2-3, to=2-4]
	\arrow["\subset", from=2-4, to=2-5]
\end{tikzcd}\]

In his formulation of the central conductor formula, Nickel defined
\[d'_\chi\colonequals \Q^{F_\chi}(\Gamma'_\chi) \cap \D\left(\M\epsilon_\chi / \Lambda^{\OO_{F_\chi}}(\Gamma'_\chi)\right);\]
this is a fractional ideal of $\Lambda^{\OO_{F_\chi}}(\Gamma'_\chi)$, independent of the choice of $\M$. The formula for the central conductor {\cite[Theorem~3.5]{NickelConductor} reads}
\[\F\left(\M / \Lambda^{\OO_F}(\G)\right) = \bigoplus_{\chi\in\Irr(\G)/\sim_F} \left(\frac{\#H w_\chi}{\chi(1)}\right) \D(\OO_{F_\chi}/\OO_F) d'_\chi.\]
Transferring this to the matrix ring side of the Wedderburn isomorphism, we define
\begin{equation} \label{eq:dpp-def}
    d''_\chi\colonequals \Q^{F_\chi}(\Gamma''_\chi) \cap \D\left(\phi_\chi(\M\epsilon_\chi) / \Lambda^{\OO_{F_\chi}}(\Gamma''_\chi)\right).
\end{equation}
We obtain the following reformulation of the central conductor formula:

\begin{noproof}{proposition} \label{central-conductor-formula-double}
    For a maximal $\Lambda^{\OF}(\Gamma_0)$-order $\M$ containing $\Lambda^{\OF}(\G)$, we have
    \[\F\left(\M / \Lambda^{\OO_F}(\G)\right) \simeq \bigoplus_{\chi\in\Irr(\G)/\sim_F} \left(\frac{\#H w_\chi}{\chi(1)}\right) \D(\OO_{F_\chi}/\OO_F) d''_\chi. \qedhere\]
\end{noproof}
{In view of \cref{central-conductor-independence}, the central conductor formula also holds for all graduated orders.}

\subsection{Explicit central conductor formula} \label{sec:central-conductor}
Let $\pi_\chi$ be a uniformiser of $\OO_{F_\chi}$, and let $\pp'_\chi\colonequals\pi_\chi\Lambda^{\OO_{F_\chi}}(\Gamma'_\chi)$. Nickel showed that $d'_\chi=(\pi'_\chi)^{r_\chi}$, where $r_\chi\le 0$. Setting $\pi''_\chi\colonequals\pi_\chi\Lambda^{\OO_{F_\chi}}(\Gamma''_\chi)$, we have that $d''_\chi=(\pp''_\chi)^{r_\chi}$ with the same exponent $r_\chi$.

{We now compute $r_\chi$. The next result is a generalisation of \cite[Theorem~3.9]{NickelConductor} that drops the restriction that $\G$ is a pro-$p$ group,} and indeed the proof follows that of loc.cit. The key ingredient is the explicit description of the Wedderburn decomposition of $\Q^F(\G)$ from \cref{thm:Wedderburn}.

\begin{proposition} \label{thm:rchi} The exponent of $\pp''_\chi$ in {$d''_\chi$} is
    \[r_\chi = - \left\lfloor \frac{ v_{W}(\DD(\OO_W / \OO_{F_{\chi}})) }{ e(F(\eta)/F_\chi) } \right\rfloor ,\]
    {where $e(F(\eta)/F_\chi)$ denotes the ramification index of $F(\eta)/F_\chi$.}
\end{proposition}
\begin{remark} \label{rem:comparison}
    \Cref{thm:rchi} specialises to Nickel's result \cite[Theorem~3.9]{NickelConductor} {in the case that $\G$} is a pro-$p$ group and $F=\QQ_p$. Indeed, in this case, Nickel showed that \[r_\chi=-\left\lfloor \frac{v_{\QQ_p(\eta)}(\DD(\OO_{\QQ_p(\eta)}/\OO_{\QQ_{p,\chi}}))}{s_\chi} \right\rfloor. \]
    Since $\G$ is a pro-$p$-group, the extension $\QQ_p(\eta)/\QQ_{p,\chi}$ is a {subextension} of $\QQ_p(\zeta_{p^m})/\QQ_p$ for some $m\ge0$. In particular, $\QQ_p(\eta)/\QQ_{p,\chi}$ is totally ramified in this case, {so $e(\QQ_p(\eta)/\QQ_{p,\chi})=w_\chi/v_\chi$}. 
    The field $W$ is the unique unramified extension of $\QQ_p(\eta)$ of degree $s_\eta$, and this Schur index is $1$ due to a theorem of Schilling, Witt and Roquette \cite[Theorem~74.15]{CRII}, so $W=\QQ_p(\eta)$. It is immediate that the numerators are equal, and the denominators agree by {\cite[Theorem~4.8(ii)]{W}, which states that $s_\chi=s_\eta w_\chi/v_\chi$}.

    We also easily recover the cases {when} $\G=H\times \Gamma$ and {when} $p\nmid s_\chi$, where Nickel showed $r_\chi=0$ in \cite[Theorem~3.5(i,ii)]{NickelConductor}. In both cases, $W/F_\chi$ is unramified, so the numerator is the valuation of the different of an unramified extension, which is zero. Indeed, if $\G$ is a direct product, then $F(\eta)=F_\chi$. If $p\nmid s_\chi$, then {$F(\eta)=F_\chi$ because $F(\eta)/F_\chi$ is a $p$-extension by \cite[Lemma~1.1]{NickelConductor} and $s_\chi=s_\eta [F(\eta):F_{\chi}]$.}
\end{remark}

\begin{proof}[{Proof of \cref{thm:rchi}}]
    {The main part of the proof is computing the image of the reduced trace map on $(\Omega_\chi)_{\pp''_\chi}$.}
    Let $x\in \Omega_\chi$; using the description in \cref{thm:Wedderburn}, we can write this as 
    \[x=\sum_{i=0}^{s_\eta {\frac{w_\chi}{v_\chi}}} x_i \left(\pi \cdot(\gamma''_\chi)^{{v_\chi/w_\chi}}\right)^i,\] 
    where $x_i\in\Lambda^{\OO_W}(\Gamma_\chi'')$. To compute the reduced trace of $x$, we recall that there is a splitting map $\Phi: \Q^W(\Gamma''_\chi)\otimes D_\chi \to M_{s_\eta {w_\chi/v_\chi}}\left(\Q^W(\Gamma''_\chi)\right)$, see \cite[proof of Corollary~3.3]{W}. This map is given by
    \[\Phi\left((\gamma''_\chi)^{{v_\chi/w_\chi}}\right) = \begin{psmallmatrix}
        & \mathbf 1_{s_\eta} \\ && \ddots \\ &&& \mathbf 1_{s_\eta} \\ \gamma''_\chi \mathbf 1_{s_\eta}
    \end{psmallmatrix}, \hspace{.3em} \Phi(\gamma''_\chi) = \gamma''_\chi \mathbf 1_{s_\eta {\frac{w_\chi}{v_\chi}}}, \hspace{.3em} \Phi(d) = \begin{psmallmatrix}
        \phi(d) \\ &\tau (\phi(d)) \\ &&\ddots \\ &&&\tau^{{\frac{w_\chi}{v_\chi}}-1}(\phi(d))
    \end{psmallmatrix},\]
    where {$d\in D_\eta$}. Here $\mathbf 1_n$ stands for the $n\times n$ identity matrix, {$\phi: D_\eta \hookrightarrow W\otimes D_\eta\xrightarrow{\sim} M_{s_\eta}(W)$} is Hasse's splitting map, and $\tau$ acts entry-wise on $M_{s_\eta}(W)$. The map $\phi$ is induced by
    \[\phi(\pi)=\begin{psmallmatrix}
        & 1 \\ &&\ddots \\ &&& 1 \\ \pi_\eta
    \end{psmallmatrix}, \quad \phi(w)=\begin{psmallmatrix}
        w \\ &\sigma(w) \\ &&\ddots \\ &&&\sigma^{s_\eta-1}(w)
    \end{psmallmatrix}\]
    for $w\in W$, {where notation is as in \eqref{eq:D-eta-cyclic}}. We are now ready to compute the reduced trace. To distinguish between traces associated with field extensions and traces of matrices, we will write $\Tr_{\mathrm{mat}}$ for the latter. {The reduced trace is}
    \begin{align*}
        \tr_{D_\chi / \cent(D_\chi)} (x) &= \sum_{i=0}^{{\frac{w_\chi}{v_\chi}}} \tr_{D_\chi/\cent(D_\chi)}\left(x_i \cdot \left(\pi \cdot(\gamma''_\chi)^{{v_\chi/w_\chi}}\right)^i\right) \\
        &= \sum_{i=0}^{{\frac{w_\chi}{v_\chi}}} \Tr_{\mathrm{mat}}\left(\Phi(x_i) \cdot \Phi(\pi)^i \cdot \Phi\left((\gamma''_\chi)^{{v_\chi/w_\chi}}\right)^i\right) = \Tr_{\mathrm{mat}}(\Phi(x_0)){.}
        \intertext{Indeed, the matrix $\Phi(x_i) \cdot \Phi(\pi)^i \cdot \Phi\left((\gamma''_\chi)^{{v_\chi/w_\chi}}\right)^i$ has zero blocks in the diagonal unless $i=0$: this is because the first two terms are block diagonal matrices, and because of the shape of the term $\Phi\left((\gamma''_\chi)^{{v_\chi/w_\chi}}\right)^i$. To proceed, let us write $x_0=\sum_{j=0}^\infty y_j (\gamma''_\chi)^j$ with $y_j\in \OO_W$. Unraveling the definition of $\Phi$, we obtain:}
        &= \sum_{j=0}^\infty \sum_{k=0}^{{\frac{w_\chi}{v_\chi}}-1} \sum_{\ell=0}^{s_\eta-1} \tau^k(\sigma^\ell(y_j)) \cdot (\gamma''_\chi)^j = \sum_{j=0}^\infty \Tr_{W/F_\chi}(y_j) \cdot (\gamma''_\chi)^j.
    \end{align*}
    In the last step, we used the fact that 
    \[{\Gal\left(W/F_\chi\right)\simeq \Gal\left(W/F(\eta)\right)\times \Gal\left(F(\eta)/F_\chi\right) = \left\langle\sigma\right\rangle \times \left\langle\tau\right\rangle.}\]

    For the reduced trace of the localisation of the maximal order, we obtain the following:
    \begin{align*}
        \tr_{D_\chi/\cent(D_\chi)}\left((\Omega_\chi)_{\pp''_\chi}\right) &= \Tr_{\Q^W(\Gamma''_\chi) / \Q^{F_\chi}(\Gamma''_\chi)} \left(\Lambda^{\OO_W}(\Gamma''_\chi)_{\pp''_\chi}\right) \\
        &= \Tr_{W/F_\chi}(\OO_W) \cdot \Lambda^{\OO_{F_\chi}}(\Gamma''_\chi)_{\pp''_\chi} \\
        &= \mm^{{t_\chi}}_{{F_\chi}} \cdot \Lambda^{\OO_{F_\chi}}(\Gamma''_\chi)_{\pp''_\chi}{.}
    \end{align*}
    Here ${t_\chi}\in\ZZ$ is the maximal exponent such that
    \begin{equation} \label{eq:t-formula}
        {\OO_W \subseteq \mm^{{t_\chi}}_{F_\chi} \DD(\OO_W / \OO_{F_\chi})^{-1} = \mm_W^{e{t_\chi}-v_W(\DD(\OO_W / \OO_{F_\chi}))}}
    \end{equation}
    holds, where $\mm$ denotes maximal ideals {in the respective local fields}, and $e=e(W/F_\chi)=e(F(\eta)/F_\chi)$ is the index of ramification.

    {We can now compute the exponent of $\pp_\chi''$ in $d_\chi''$. By the definition of $d''_\chi$, see \eqref{eq:dpp-def}, an element $x\in \Q^{F_\chi}(\Gamma''_\chi)$ is contained in $(d''_\chi)_{\pp_\chi''}$ if and only if 
    \[\tr_{D_\chi / \Q^{F_\chi} (\Gamma_\chi'') } \left(x \cdot (\Omega_\chi)_{\pp_\chi''}\right) \subseteq \Lambda^{\OO_{F_\chi}}(\Gamma_\chi'')_{\pp_\chi''}.\]
    Therefore $r_\chi=-{t_\chi}$, and the assertion now follows from \eqref{eq:t-formula}.}
\end{proof}

Using the central conductor formula, Nickel proved annihilation results for Ext-groups and Fitting ideals, see \cite[\S4]{NickelConductor}. Computing $r_\chi$ as in \cref{thm:rchi} also makes these results more explicit.

\section{Equivariant \texorpdfstring{$p$}{p}-adic Artin conjecture for graduated orders} \label{sec:epac}
The $p$-adic Artin conjectures are originally due to Greenberg \cite[pp.~82, 87]{Greenberg-padicArtinL}, and they were proven by Wiles \cite[Theorem~1.1]{Wiles} and Ritter--Weiss \cite[Remark~(G)]{TEIT-II}. The following equivariant analogue has been proposed by the author \cite{EpAC}. 

Let $p$ be an odd rational prime, $K$ a totally real number field, and $\mathcal L/K$ an admissible one-dimensional $p$-adic Lie extension: that is, $\mathcal L$ is totally real, contains the cyclotomic $\ZZ_p$-extension $K_\infty$ of $K$, and the degree $[\mathcal L:K_\infty]$ is finite.
Let $\G\colonequals\Gal(\mathcal L/K)$ be the Galois group; then $\G\simeq H\rtimes\Gamma$ where $H$ is a finite group and $\Gamma\simeq\ZZ_p$. Let $n_0\gg0$ be large enough such that $\Gamma_0\colonequals\Gamma^{p^n}$ is central in $\G$. Let $S$ be a set of places of $K$ containing all places ramifying in $\mathcal L/K$ and all infinite places, and let $T$ be a non-empty finite set of places of $K$ disjoint to $S$. Then there is a smoothed equivariant $p$-adic Artin $L$-function $\Phi_S^T$ attached to the extension $\mathcal L/K$; we refer to {\cite[\S2]{EpAC}} for its definition.
\begin{conjecture}[equivariant $p$-adic Artin conjecture] \label{integrality-Phi}
	Let $\mathfrak M$ be a $\Lambda(\Gamma_0)$-order in $\Q(\G)$ containing $\Lambda(\G)$. Then the smoothed equivariant $p$-adic Artin $L$-function $\Phi_S^T$ is in the image of the composite map
	\begin{equation} \label{nr-composite}
		\mathfrak M\cap \Q(\G)^\times \to K_1(\Q(\G)) \xrightarrow{\nr} \cent(\Q(\G))^\times,
	\end{equation}
	where the first arrow is the natural map sending an invertible element to the class of the $1\times 1$ matrix consisting of said element.
\end{conjecture}
Assuming the equivariant Iwasawa main conjecture {as stated in \cite[Conjecture~4.3]{NABS}}, the equivariant $p$-adic Artin conjecture has been proven for maximal orders, see {\cite[Theorem~7.4]{EpAC}}. 
Note that the main conjecture has been verified in multiple interesting cases, such as when the $\mu$-invariant vanishes \cite{RW-MC,Kakde}, or when $\G$ has an abelian $p$-Sylow subgroup \cite[Corollary 12.17]{UAEIMC}.
Using the main conjecture, the proof of \cref{integrality-Phi} for maximal orders is reduced to the following statement, which is an adaptation of \cite[Proposition~2.13]{NichiforPalvannan} to semidirect products $\G$:
\begin{proposition}[{\cite[Proposition~6.2]{EpAC}}] \label{2.13}
	Let $m\ge1$, $n\ge1$, $\chi\in\Irr(\G)$, and
	\[\alpha \in M_m\left( \Omega_\chi \right) \cap \GL_m \left( D_\chi \right).\]
	Then the Dieudonné determinant of $\alpha$ is integral: $\det(\alpha)\in \Omega_\chi\cap D_\chi^\times$. \qed
\end{proposition}

\begin{corollary}
    Assume the equivariant Iwasawa main conjecture without uniqueness. Then \cref{integrality-Phi} holds for all graduated $\Lambda(\Gamma_0)$-orders $\mathfrak M$ in $\Q(\G)$ containing $\Lambda(\G)$.
\end{corollary}
\begin{proof}
    We may work Wedderburn componentwise. First suppose that the image of $\mathfrak M\epsilon_\chi$ under $\mathfrak M\epsilon_\chi\subset \Q(\G)\epsilon_\chi\simeq M_{n_\chi}(D_\chi)$ is in standard form. By the same argument as in the proof of \cite[Theorem~7.4]{EpAC}, the assertion reduces to showing that if $\alpha_{{\chi}}\in M_{m}(\mathfrak M \epsilon_\chi)\cap \GL_{mn_\chi}(D_\chi)$, then there is an $A_\chi\in \mathfrak M\epsilon_\chi\cap \GL_{n_\chi}(D_\chi)$ with $\det(\alpha_\chi)=\det(A_\chi)$. But $A_\chi\colonequals \diag(\det(\alpha_{{\chi}}),1,\ldots,1)$ clearly satisfies this condition by \cref{2.13}.
    {If the image of $\mathfrak M\epsilon_\chi$ is not in standard form, then by \cref{every-standard-form}, the image is $\mathfrak M\epsilon_\chi\simeq u_\chi \LLambda u_\chi^{-1}$} for some graduated $R$-order $\LLambda$ in standard form and $u_\chi\in\GL_{n_\chi}(D_\chi)$, {and} then the same argument as in \cite[Corollary~7.6]{EpAC} works.
\end{proof}

\printbibliography
\end{document}